\documentclass[11pt]{article}
\usepackage{amsmath,amsfonts,amsthm,alltt,amssymb,fullpage}
\usepackage[varg]{pxfonts}
\usepackage{microtype}

\allowdisplaybreaks[2]

\makeatletter
\newtheorem*{rep@theorem}{\rep@title}
\newcommand{\newreptheorem}[2]{%
\newenvironment{rep#1}[1]{%
 \def\rep@title{#2 \ref{##1}}%
 \begin{rep@theorem}}%
 {\end{rep@theorem}}}
\makeatother

\newtheorem{theorem}{Theorem}[section]
\newreptheorem{theorem}{Theorem}
\newtheorem{corollary}[theorem]{Corollary}
\newtheorem{lemma}[theorem]{Lemma}

\newtheorem{proposition}[theorem]{Proposition}
\newtheorem{claim}[theorem]{Claim}

\theoremstyle{definition}
\newtheorem{definition}[theorem]{Definition}
\newtheorem{remark}[theorem]{Remark}
\newtheorem{example}[theorem]{Example}

\newdimen\pIR
\newcommand\StevesR{{\rm I\kern\pIR R}}
\def\Reals#1{\StevesR^{#1}}

\newcommand\symExpec{\operatorname{\mathbb{E}}\displaylimits}

\def\expec#1{\symExpec_{#1} }

\newcommand{\E}{\mathbb{E}}

\def\defeq{\stackrel{\mathrm{def}}{=}}
\def\setof#1{\left\{#1  \right\}}

\def\aa{\pmb{\mathit{a}}}
\newcommand\bb{\boldsymbol{\mathit{b}}}
\newcommand\cc{\boldsymbol{\mathit{c}}}
\newcommand\dd{\boldsymbol{\mathit{d}}}

\newcommand\nn{\boldsymbol{\mathit{n}}}

\newcommand\mm{\boldsymbol{\mathit{m}}}

\newcommand\phat{{\widehat{{p}}}}

\newcommand\ptil{{\widetilde{{p}}}}

\def\Branden{Br\"{a}nd\'{e}n}

\newcommand{\dan}[1]{}
\newcommand{\nikhil}[1]{}
\newcommand{\adam}[1]{}

\newcommand{\R}{\mathbb{R}}
\newcommand{\AND}{\quad \text{and} \quad}
\newcommand{\mydet}[1]{\det\left(#1\right)}

\newcommand{\smax}[2]{#1\text{--}\mathsf{max}\left(#2 \right)}
\newcommand{\maxroot}[1]{\mathsf{maxroot} \left(#1 \right)}

\usepackage{tikz,amstext}

\newcommand*\myrect[1]{%
  \begin{tikzpicture}
    \node[draw,rectangle,inner sep=0pt] {#1};
  \end{tikzpicture}}

\newcommand{\sqsum}{\mathbin{\text{\myrect{\textnormal{+}}}}}
\newcommand{\recsum}{\mathbin{\textnormal{\myrect{++}}}}
\newcommand{\sqmult}{\mathbin{\textnormal{\myrect{$\times$}}}}
\newcommand{\cauchy}[2]{\mathcal{G}_{#1} \left(#2 \right)}
\newcommand{\invcauchy}[2]{\mathcal{K}_{#1} \left(#2 \right)}
\newcommand{\rtrans}[2]{\mathcal{R}_{#1} \left(#2 \right)}

\newcommand{\strans}[2]{\widetilde{\mathcal{S}}_{#1} \left(#2 \right)}
\newcommand{\mtrans}[2]{\widetilde{\mathcal{M}}_{#1} \left(#2 \right)}
\newcommand{\invmtrans}[2]{\widetilde{\mathcal{M}}^{(-1)}_{#1} \left(#2 \right)}

\newcommand{\subsquare}{\mathbb{S}}
\newcommand{\deriv}{D}
\newcommand{\dxd}{\deriv x \deriv}
\newcommand{\shiftoper}[1]{U_{#1}}

\newcommand{\multoper}[1]{V_{#1}}

\newcommand{\rrpoly}[1]{\mathbb{P} (#1)}
\newcommand{\RRpoly}{\mathbb{P}}
\newcommand{\rrpos}[1]{\mathbb{P}^{+} (#1)}
\newcommand{\RRpos}{\mathbb{P}^{+}}
\newcommand{\rrneg}[1]{\mathbb{P}^{-} (#1)}
\newcommand{\RRneg}{\mathbb{P}^{-}}

\newcommand{\charp}[2]{\chi_{#2} \left(#1 \right)}

\newcommand{\csteifel}[2]{\mathbb{C}^{#1}_{#2}}

\begin{document}

\title{
Finite free convolutions of polynomials
\thanks{
This research was partially supported by NSF grants CCF-1111257, CCF-1562041, 
  an NSF Mathematical Sciences Postdoctoral Research Fellowship, Grant No. DMS-0902962,
  a Simons Investigator Award to Daniel Spielman, and a MacArthur Fellowship.
}}

\author{
Adam W. Marcus\\
Princeton University\\
\and
Daniel A. Spielman \\ 
Yale University
\and 
Nikhil Srivastava\\
UC Berkeley
}

\maketitle

\begin{abstract}
We study three convolutions of polynomials in the context of free probability theory.
We prove that these convolutions can be written as the expected characteristic polynomials of sums and products of unitarily invariant random matrices.
The symmetric additive and multiplicative convolutions were introduced by Walsh and Szeg{\"o} in different contexts, and have been studied for a century.
The asymmetric additive convolution, and the connection of all of them with random matrices, is new.
By developing the analogy with free probability, we prove that these convolutions produce real rooted polynomials and provide strong bounds
  on the locations of the roots of these polynomials.
\dan{Turn this off}
\end{abstract}

\tableofcontents

\section{Introduction}
We study three convolutions on polynomials that are inspired by free probability theory.\nikhil{is inspired still appropriate?}\dan{I think so.}
Instead of capturing the limiting spectral distributions of ensembles of random matrices, we show that
  these capture the expected characteristic polynomials of random matrices in a fixed dimension.
We develop the analogy with Free Probability by proving that
   Voiculescu's $R$ and $S$-transforms can be used to prove upper bounds on the extreme roots of these polynomials.
Two of the convolutions have been classically studied.
The third, and the connection of all of them with random matrices, is new.
We begin by defining the three convolutions and stating the algebraic identities that establish this connection, as well
  as basic results regarding their real-rootedness properties.
\subsection{Algebraic Identities and Real Rootedness}
\subsubsection*{Symmetric Additive Convolution}
\begin{definition}\label{def:sqsum}
For complex univariate polynomials
\begin{align*}
  p (x) = \sum_{i=0}^{d} x^{d-i} (-1)^{i} a_{i}
\quad  \text{ and} \quad 
  q (x)  = \sum_{i=0}^{d} x^{d-i} (-1)^{i} b_{i}
\end{align*}
of degree at most $d$, the {\em $d$th symmetric additive convolution} of $p$ and $q$ is defined as:
\begin{align}\label{eqn:symmetricSum}
	p (x) \sqsum_{d} q (x)  &\defeq \sum_{k=0}^{d} x^{d-k} (-1)^{k} 
	\sum_{i+j = k} \frac{(d-i)! (d-j)!}{d! (d-k)!}   a_{i} b_{j} 
	\\&= \frac{1}{d!}\sum_{k=0}^d D^k p(x) D^{d-k} q(0) \nonumber
	\\&= \frac{1}{d!}\sum_{k=0}^d D^k q(x) D^{d-k} p(0), \nonumber
\end{align}
\end{definition}
\noindent where $D$ denotes differentiation with respect to $x$. 
Note that we have defined a {\em sequence} of convolutions parameterized by $d$, and in general $p\sqsum_c q \neq p\sqsum_{d} q$, even if both
$p$ and $q$ have degree less than $c$ and $d$; we will discuss this point more in Section \ref{sec:diffdegs}.

Observe that the definition above has the compact form:
\begin{equation}\label{eqn:sqsumcompact}
(p\sqsum_d q)(x) = \hat{p}(D)\hat{q}(D)x^d,\end{equation}
where $\hat{p}$ and $\hat{q}$ are the unique polynomials satisfying $\hat{p}(D)x^d=p(x)$ and $\hat{q}(D)x^d=q(x)$.
This reveals that $\sqsum_d$ is symmetric, bilinear in its arguments, and commutes with differentiation and translation; i.e.,
$$ (Dp(x))\sqsum_d q(x) = D(p(x)\sqsum_d q(x))\qquad\textrm{and}\qquad p(x-t)\sqsum_d q(x) = (p\sqsum_d q)(x-t).$$ 
Note that the identity element for $\sqsum_d$ is the polynomial $x^d$.

For a square $d\times d$ complex matrix $M$, we define 
\[
\charp{M}{x} \defeq \det (xI - M),
\]
its characteristic polynomial in the variable $x$.
We show that for monic polynomials, the operation $\sqsum_d$ can be realized as an expected characteristic polynomial
of a sum of random matrices.
\begin{theorem}\label{thm:symmetricAdditive} If $p(x)=\charp{A}{x}$ and $q(x)=\charp{B}{x}$ for $d\times d$ complex normal matrices $A,B$, then
	\begin{equation}\label{eqn:sqsummat}  p (x) \sqsum_{d} q (x) = \expec{Q} \charp{A + Q B Q^{*}}{x},\end{equation}
where the expectation is over a random unitary matrix $Q$ from the Haar measure on $U(d)$.
\end{theorem}
   In fact, by unitary invariance the right hand side depends only on 
   $\charp{A}{x}$ and $\charp{B}{x}$ and not on the further details of $A$ and
   $B$, so we may take them to be any normal matrices with the same characteristic polynomials.

The convolution \eqref{eqn:symmetricSum} was studied by Walsh~\cite{Walsh22}, who proved results
  including the following theorem (see also \cite{Marden} and \cite[Theorem 5.3.1]{rahman2002analytic}).
\begin{theorem}\label{thm:symmetricAdditiveRR}
If $p$ and $q$ are real rooted polynomials of degree $d$, then so is
  $p \sqsum_{d} q$.
Moreover, 
\[
  \maxroot{p \sqsum_{d} q} \leq \maxroot{p} + \maxroot{q}.
\]
\end{theorem}
In Theorem~\ref{thm:sqsumTrans} we strengthen this bound on the maximum root.
Our result is much tighter in the case that most of the roots of $p$ and $q$
  are far from their maximum roots.

\subsubsection*{Symmetric Multiplicative Convolution}
\begin{definition} \label{def:sqmult}
For complex univariate polynomials
\begin{align*}
  p (x)  = \sum_{i=0}^{d} x^{d-i} (-1)^{i} a_{i} 
 \quad \text{ and} \quad 
  q (x)  = \sum_{i=0}^{d} x^{d-i} (-1)^{i} b_{i}
\end{align*}
of degree at most $d$, the {\em $d$th symmetric multiplicative convolution} of $p$ and $q$ is defined as:
	\begin{equation} \label{eqn:symmetricMult}
  p (x) \sqmult_{d} q (x) \defeq
  \sum_{i=0}^{d} x^{d-i} (-1)^{i} \frac{a_{i} b_{i}}{\binom{d}{i}}.
\end{equation}
\end{definition}
It is clear that $\sqmult_d$ is also bilinear, though it does not commute with differentiation or translation.

The following compact differential form of \eqref{eqn:symmetricMult}, analogous to \eqref{eqn:sqsumcompact}, was discovered by B. Mirabelli \cite{benno} who has kindly allowed us to include it here: if
$p(x)=P(xD)(x-1)^{d}$ and $q(x)=Q(xD)(x-1)^{d}$, then
\begin{equation}\label{eqn:sqmultcompact}
	(p\sqmult_d q)(x) = P(xD)Q(xD)(x-1)^d.
\end{equation}
Note that every polynomial of degree at most $d$ can be written as $P(xD)(x-1)^d$, though it is not as obvious as in the additive case.
The appearance of the polynomial $(x-1)^d$ is explained by the fact that it is the identity element for $\sqmult_d$, i.e., $p(x)\sqmult_d (x-1)^d=p(x)$
for every $p$ of degree at most $d$. 

We show that for monic polynomials the operation $\sqmult_d$ can be realized as an expected characteristic polynomial of a {\em product}
of random matrices.
\begin{theorem} \label{thm:symmetricMult}
	If $p(x)=\charp{A}{x}$ and $q(x)=\charp{B}{x}$ for $d\times d$ complex normal matrices $A,B$, then
	\begin{equation}\label{eqn:sqmultmat}
  p (x) \sqmult_{d} q (x) = \expec{Q} \charp{A Q B Q^{*}}{x},
\end{equation}
where the expectation is over a Haar unitary $Q$. 
\end{theorem}
\noindent The identity element $(x-1)^d$ thereby corresponds to taking $B=I$ in the above formula.

This convolution was studied by Szeg\"o \cite{Szego22}, who proved the following theorem.
\begin{theorem}\label{thm:symmetricMultRR}
If $p$ and $q$ have only nonnegative real roots, then the same is true of $p \sqmult_{d} q$.
Moreover,
\[
    \maxroot{p \sqmult_{d} q} \leq \maxroot{p} \maxroot{q}.
\]
\end{theorem}
We strengthen this result in Theorem~\ref{thm:recsumTrans}.

\subsubsection*{Asymmetric Additive Convolution}
\begin{definition}\label{def:recsum}
For complex univariate polynomials
\[
  p (x)  = \sum_{i=0}^{d} x^{d-i} (-1)^{i} a_{i},
\quad \text{and} \quad 
  q (x)  = \sum_{i=0}^{d} x^{d-i} (-1)^{i} b_{i},
\]
of degree at most $d$, we define the {\em $d$th asymmetric additive convolution} of $p$ and $q$ to be:
\begin{equation}\label{eqn:recsummat}
  p (x) \recsum_{d} q (x) = \sum_{k=0}^{d} x^{d-k} (-1)^{k} 
\sum_{i+j = k}
 \left(\frac{(d-i)! (d-j)!}{d! (d-k)!} \right)^{2}   a_{i} b_{j}.
\end{equation}
\end{definition}
This is equivalent to the expression
\[
  p (x) \recsum_{d} q (x) 
 =
  \left(\frac{1}{d!} \right)^{2} \sum_{i=0}^{d} ((d-i)!)^{2} b_{i} (\deriv x \deriv)^{i} p (x),
\]
which may also be written for $p(x)=P(DxD)x^{d}$ and $q(x)=Q(DxD)x^{d}$ as 
$$ p(x)\recsum_d q(x) = P(DxD)Q(DxD)x^{d},$$
the latter being an observation of \cite{benno}.

We are not aware of previous studies of this convolution. We show that if $p(x)$ and $q(x)$ are monic {\em with all roots real and nonnegative},
then $p \recsum_d q$ can be realized as an expected characteristic polynomial.
\begin{theorem}\label{thm:asymmetricAdditive}
For two monic polynomials $p(x)=\charp{AA^*}{x}$ and $q(x)=\charp{BB^*}{x}$ with $A,B$ any $d\times d$ square complex matrices,
\begin{equation}
	p (x) \recsum_{d} q (x) = \expec{R,Q} \charp{(A + R B Q) (A + R B Q)^{*}}{x}.
\end{equation}
where the expectation is taken over independent Haar unitaries $R$ and $Q$.
\end{theorem}
While it is not immediately obvious, we show in Theorem \ref{thm:asymmetricAdditiveCoeff}
that the right hand side depends only on $p(x)$ and $q(x)$.
We prove the following real-rootedness theorem in Section \ref{sec:rr}.
\begin{theorem}\label{thm:asymAdditiveRR}
If $p$ and $q$ have only nonnegative real roots, then the same is true of
  $p \recsum_{d} q$.
\end{theorem}
We obtain bounds on the roots of $p\recsum_d q$ in Theorem \ref{thm:rectConvBound}.

\begin{remark} In the first version of this paper \cite{FFCold}, we {\em defined} the operations $\sqsum_d, \sqmult_d,$ and $\recsum_d$ in terms
	of random matrices using the formulas \eqref{eqn:sqsummat}, \eqref{eqn:sqmultmat}, \eqref{eqn:recsummat}, and stated the 
	formulas appearing in Definitions \ref{def:sqsum}, \ref{def:sqmult}, and \ref{def:recsum} as theorems by showing that they
	were equivalent. We have chosen to reverse this presentation since $\sqsum_d,\sqmult_d$ were in fact already defined by Walsh
	and Szeg\"o, albeit in a different context, and their basic properties such as bilinearity are more immediate from the purely algebraic definitions.
\end{remark}
\subsection{Motivation and Related Results}
Before describing our analytic results on the locations of roots of the convolutions in the next section, we
explain the motivations for studying them in the context of several other areas of mathematics.

\subsubsection*{Free Probability} Free probability theory (see e.g.
\cite{voiculescu1997free,nica2006lectures,anderson2010introduction}) studies among other things
  the large $d$ limits of random matrices such as those in \eqref{def:sqsum}, \eqref{def:sqmult}, and \eqref{def:recsum}.
In particular, it allows one to calculate the limiting spectral distribution of a 
  sum or product of two unitarily invariant random matrix ensembles in terms of the limiting spectral
  distributions of the individual ensembles.
For both the sum and the product, free probability provides a transform of the moments of the spectra of the individual
  matrices from which one can easily derive the transform of the moments of the limiting spectra of the resulting matrices --- these
  are known as the $R-$ and $S-$transforms, which may be viewed as generating functions for certain polynomials in the 
  moments (known as free cumulants) which are linear in the convolutions.
We show that for our ``finite free convolutions'' the same transforms provide {\em upper bounds} on the roots of the corresponding
  expected polynomials in {\em finite dimensions}.

Following the definitions in this paper, \cite{adam,arizmendi} have shown that by taking appropriate
  limits our finite free convolutions yield the standard free convolutions in free probability theory. 
Thus, expected characteristic polynomials provide an alternative ``discretization'' of free convolutions from the typical
  one involving random matrices.

The paper \cite{gorin} shows that the real zeros of $p\sqsum_d q$ and $p\sqmult_d q$ may be interpreted as the $\beta\rightarrow \infty$ limit of 
certain generalizations of $\beta-$ensembles in random matrix theory. 

\subsubsection*{Combinatorics} The original motivation for this work is the method of interlacing families of polynomials
  \cite{IF1,IF2,IF4}, which reduces various combinatorial problems concerning eigenvalues
  to problems of bounding the roots of the expected characteristic polynomials of certain random matrices.
In particular, the paper \cite{IF4} studies bipartite random regular graphs, whose expected characteristic polynomials
  turn out to be of the type appearing in \eqref{def:recsum}. The bound of Theorem \ref{thm:multTransBound} on the roots of these polynomials then implies the existence of bipartite Ramanujan graphs of every size and every degree (a result that was later turned
into a polynomial time algorithm by Cohen \cite{cohen}).  Hall, Puder, and Sawin \cite{HPS} used some of the techniques in this paper
to prove the related result that every bipartite graph has a $k-$cover which is Ramanujan, for every $k\ge 2$, generalizing \cite{IF1}.

\subsubsection*{\bf Representation Theory} As shown in Section \ref{sec:formulas},
the unitary group may be replaced by the orthogonal group or the group of
signed permutations in Theorems
\ref{thm:symmetricAdditive}, \ref{thm:symmetricMult}, and \ref{thm:asymmetricAdditive}
without changing the expected characteristic polynomial and therefore any of
their statements.  The ability to compute the average of a matrix function over
the group of unitary matrices by instead computing an average over some
smaller group of matrices is a phenomenon we refer to as {\em quadrature} (see \cite{IF4} for more details).

The proofs given Section \ref{sec:formulas}  are different from those that
appeared in the first version \cite{FFCold} of this paper.  The original proofs treat each
of the convolutions as (specifically) being an average over the unitary group,
and then by explicit calculations show that one can (in some cases) replace
the unitary group with the signed permutation matrices and get the same result.
After posting that preprint, we were informed that our quadrature results
were actually a manifestation of well-studied concepts in representation
theory (the Peter-Weyl theorem) and further work has characterized all
subgroups of the unitary group that have this property \cite{HPS}.

One subgroup that has been used specifically in an application is the $n\times
n$ matrices corresponding to the standard representation of $S_{n+1}$ (the
symmetric group on $n+1$ elements).  The fact that this group is a valid
quadrature group plays a pivotal role in the results on Ramanujan graphs
\cite{IF4} and \cite{HPS} mentioned above.

\subsubsection*{\bf Geometry of Polynomials} Theorem \ref{thm:symmetricAdditiveRR} implies that if $p(x)$ is real-rooted of degree $d$, then
the linear transformation $p\sqsum_d (\cdot)$ preserves real-rootedness of all real polynomials of degree at most $d$. 
Leake and Ryder \cite{leakeryder} observe that a partial converse is also true: every differential operator $T:\R_{\le d}[x]\rightarrow \R_{\le d}[x]$ 
preserving real-rootedness can be written as $T(q)=p\sqsum_d q$ for some $p$ of degree at most $d$. Thus, our bounds on
the extreme roots of the additive convolution imply bounds on how much any such operator can perturb the roots of its real rooted inputs. They also generalize our analytic bounds on the roots of symmetric additive convolutions (Theorem \ref{thm:sqsumTrans})
  by showing that they are a special case of submodularity relation.

\subsubsection*{\bf Random Matrix Theory}
The expected characteristic polynomials of symmetric Gaussian random matrices are Hermite polynomials (see, e.g. \cite[Theorem 4.1]{dumitriu2002matrix}).
If $R$ is an $d$-by-$d$ matrix of independent Gaussian random variables of variance $1$, then
\[
	\E \charp{\frac{R + R^T}{\sqrt{2}}}{x} = H_d (x),  
\]
where 
\[
  H_d (x) = e^{-D^2 / 2} x^d
\]
is the $d$th Hermite polynomial.
By applying Theorem \ref{thm:symmetricAdditive}, but taking the expectation over orthogonal matrices, or by applying the formula \eqref{eqn:sqsumcompact},
we may conclude that for positive $a$ and $b$ and $c = \sqrt{a^2 + b^2}$,
\[
  a^d H_d (x / a) \sqsum_d b^d H_d (x / b) = c^d H_d (x/c).
\]

Similarly, the expected characteristic polynomial of $R R^T$ is the $d$th Laguerre polynomial \cite[Section 9]{edelman1988eigenvalues}
\[
L_d (x) = \left(1 - D \right)^d x^d.  
\]
Thus, both Theorem \ref{thm:asymmetricAdditive} and the definition \eqref{eqn:recsummat} can be used to show that for postive $a$ and $b$ and $c = a + b$,
\[
  a^d L_d (x/a) \recsum_{d}  b^d L_d (x/b) = c^d L_d (x/c).
\]

\subsection{Transforms and Root Bounds}\label{sec:cauchy}
In free probability, each of the three convolutions comes equipped with a
  natural transform of probability measures.
We define analogous transforms on polynomials and use them to bound the extreme roots
  of the convolutions of polynomials.

We will identify a vector $(\lambda_{1}, \dotsc , \lambda_{d})$
  with the discrete distribution that takes each value $\lambda_{i}$
  with probability $1/d$.
The Cauchy/Hilbert/Stieltjes transform of such a distribution is the function
\[
  \cauchy{\lambda}{x} = \frac{1}{d} \sum_{i=1}^{d} \frac{1}{x-\lambda_{i}}.
\]
Given a polynomial $p$ with roots $\lambda_{1}, \dotsc , \lambda_{d}$,
  we similarly define
\[
  \cauchy{p}{x} := \cauchy{\lambda}{x}.
\]
We will prove theorems about the inverse  of the Cauchy transform, which we define for real $w>0$ by
\[
  \invcauchy{p}{w} := \max \setof{x : \cauchy{p}{x} = w}.
\]
For a real rooted polynomial $p$, and thus for real $\lambda_{1}, \dotsc , \lambda_{d}$,
 this is the value of $x$ that is larger than all the $\lambda_{i}$ for which $\cauchy{p}{x} = w$. Since
 $\cauchy{p}{x}$ is decreasing above the largest root of $p$, the maximum is uniquely attained
 for each $w>0$.

As $\cauchy{p}{x} = \frac{1}{d} \frac{p' (x)}{p (x)}$,
\[
\cauchy{p}{x} = w
\quad \iff \quad 
p (x) - \frac{1}{w d} p' (x) = 0.
\]
This tells us that 
\[
  \invcauchy{p}{w} = \maxroot{\shiftoper{1/wd} p},
\]
where we define
\[
  \shiftoper{\alpha} p (x) :=  p (x) - \alpha \deriv p (x).
\]

Voiculescu's R-transform of the probability distribution that is uniform on $\lambda$
  is given by $\rtrans{\lambda }{w} = \invcauchy{\lambda }{w} - 1/w$ (though in free
  probability the inversion is typically done at the level of power series).
We use the same notation to define a transform on polynomials
\[
  \rtrans{p}{w} := \invcauchy{p}{w} - 1/w.
\]

If $\lambda$ and $\mu$ are compactly supported probability distributions on the reals, then their
  free convolution $\lambda \sqsum \mu $ satisfies \cite{voiculescu1997free}:
\[
  \rtrans{\lambda \sqsum \mu}{w} = \rtrans{\lambda}{w} + \rtrans{\mu}{w}.
\]

For our finite additive convolution, we obtain an analogous {\em inequality} for $w>0$.
\begin{theorem}{thm:sqsumTrans}
 For $w > 0$ and real-rooted polynomials $p$ and $q$ of degree $d$,
 \begin{equation}\label{eqn:sqsumTrans}
  \rtrans{p \sqsum_{d} q}{w} \leq  \rtrans{p}{w} + \rtrans{q}{w},
\end{equation}
\end{theorem}
with equality if and only if $p(x)$ or $q(x)$ has the form $(x-\lambda)^{d}$.
We will often write \eqref{eqn:sqsumTrans} as:
 \[
    \maxroot{\shiftoper{\alpha} (p \sqsum_{d} q)} + d \alpha
    \leq 
    \maxroot{\shiftoper{\alpha} p}
    +
    \maxroot{\shiftoper{\alpha} q},
  \]
where $\alpha = 1/wd$.

To bound the roots of the finite multiplicative convolution, we employ a variant of Voiculescu's $S-$transform.
We first define a variant of the moment transform, which we write as a power series in $1/z$ instead of in $z$:
\[
  \mtrans{p}{z} := z \cauchy{p}{z} - 1
=
  \frac{1}{d} \sum_{i=1}^{d} \sum_{j \geq 1} \left(\frac{\lambda_{i}}{z} \right)^{j}.
\]
For a polynomial $p$ having only nonnegative real roots and a $z > 0$,
\[
z > \maxroot{p} \quad \iff \quad \mtrans{p}{z} < \infty .
\]
We define the inverse of this transform, $\invmtrans{p}{w}$, to be the largest
  $z$ so that $\mtrans{p}{z} = w$, and 
\[
\strans{p}{w}
 = 
\frac{w}{w+1} \invmtrans{p}{w}.
\]

This is the reciprocal of the usual $S$-transform.
We prove the following bound on this transformation in Section~\ref{sec:multTrans}.

\begin{theorem}{thm:multTransBound}
For degree $d$ polynomials $p$ and $q$ having only nonnegative real roots and $w > 0$,
\[
  \strans{p \sqmult_{d} q}{w} \leq \strans{p}{w} \strans{q}{w},
\]
with equality only when $p$ or $q$ has only one distinct root.
\end{theorem}

To define the relevant transform for the asymmetric additive convolution,
  we define $\subsquare$ to be the linear map taking a polynomial $p (x)$
  to $p (x^{2})$.
If $p$ has only positive real roots $\lambda_{i}$, then $\subsquare p$ has roots $\pm \sqrt{\lambda_{i}}$.
If $\lambda$ is a probability distribution supported on the nonnegative reals,
  then we use $\subsquare \lambda$ to denote the corresponding symmetric probability distribution on $\pm \sqrt{\lambda_{i}}$.

If $\lambda$ and $\mu$ are compactly supported probability distributions on the positive reals, then
Benaych-Georges \cite{benaych2009rectangular} showed that their appropriately defined rectangular convolution 
$\lambda \recsum \mu$ satisfies:
\[
  \rtrans{\subsquare \lambda \recsum \subsquare \mu}{w} = \rtrans{\subsquare \lambda}{w} + \rtrans{\subsquare \mu}{w}.
\]

In Section~\ref{sec:asymTrans} we derive an analog of this result in the form of the following inequality. 

\begin{theorem}\label{thm:recsumTrans}
For degree $d$ polynomials $p$ and $q$ having only nonnegative real roots and $w>0$,
\[
	\rtrans{\subsquare (p \recsum_{d} q)}{w} \leq  \rtrans{\subsquare p}{w} + \rtrans{\subsquare q}{w}.
\]
\end{theorem}

\begin{remark} The formulas above are stated only for polyomials of degree exactly $d$, but they may be
	applied to polynomials of degree at most $d$ by first applying the degree-reduction formulas outlined
in the next section.\end{remark}

\subsection{Polynomials of Different Degrees}\label{sec:diffdegs}
The operation $p\sqsum_d q$ is defined above for pairs of
polynomials $p$ and $q$ of degree at most $d$, but if one or both of the
polynomials has degree $c<d$, it may be written in terms of the lower degree operation $\sqsum_c$.
\begin{lemma}[Degree Reduction for $\sqsum_d$]\label{lem:reduce_degree} Suppose $\deg(p)\le d$ and $\deg(q)\le d-1$. Then:
	\begin{equation} p(x)\sqsum_d q(x) = (D/d)p(x)\sqsum_{d-1} q(x)\end{equation}
\end{lemma}
\begin{proof} By \eqref{eqn:sqsumcompact}, the differential operator $p(x)\sqsum_d (\cdot):\R_{d-1}[x]\rightarrow\R_{d-1}[x]$ is equal to $\hat{p}(D)$ where $\hat{p}(D)x^d=p(x)$. 
	Differentiating, we have 
	$$\hat{p}(D)x^{d-1}=(D/d)\hat{p}(D)x^d = (Dp/d)(x),$$
	so we must have $p(x)\sqsum_d q(x) = (Dp/d)(x)\sqsum_{d-1} q(x)$ whenever $\deg(q)\le d-1$.
\end{proof}
This relationship between convolutions of different grades turns out to be a key
tool in the analytic proofs in Section \ref{sec:bounds}, which induct on the degrees of the polynomials.
We prove similar degree reduction formulas for $\sqmult_d$ and $\recsum_d$ in Lemmas \ref{lem:multdegm1} and \ref{lem:rectConvDiffDegs},
but have chosen to present them in context along with the inductive proofs in Section \ref{sec:bounds}. 

It turns out that the operation $\sqsum_d$ can be naturally defined for
polynomials of degree strictly greater than $d$ via the random matrix identity \eqref{eqn:sqsummat}. The key observation is that the expected
characteristic polynomial of a random restriction of a $d\times d$ matrix
is proportional to a derivative of its characteristic polynomial.

\begin{lemma}\label{lem:restrict} If $a > d$, $A$ is an $a\times a$ matrix and $Q$ is a random $d\times a$ complex matrix with orthonormal rows (i.e., sampled from
	the Haar measure on the complex $d\times a$ Stiefel manifold $\csteifel{d}{a}$), then
$$ 
	\E_{Q} \charp{Q A Q^{*}}{x} = \frac{d!}{a!}D^{a-d}q(x).
$$
\end{lemma}
We prove this lemma in Section \ref{sec:formulas}.
Lemma \ref{lem:restrict} yields the following corollary, which may be viewed as a generalization of the definition of $\sqsum_d$ to polynomials of degree greater than $d$.
\begin{corollary}
	If $A$ and $B$ are $a\times a$ and $b\times b$ Hermitian matrices with $a,b\ge d$ and characteristic polynomials $p(x)$ and $q(x)$ respectively, and $R, Q$ are uniformly random from $\csteifel{d}{a}$ and $\csteifel{d}{b}$ respectively, then
	\begin{equation}\label{eqn:sqsumcd}
		p(x)\sqsum_d q(x) \defeq \E_{Q,R} \charp{RAR^*+QBQ^*}{x} = \frac{d!}{a!} D^{a-d}p(x)\sqsum_d \frac{d!}{b!}D^{b-d}q(x).
	\end{equation}
\end{corollary}
\begin{proof}
Since the formula \eqref{eqn:symmetricSum} is bilinear in the characteristic polynomials $\charp{A}{x}$ and $\charp{B}{x}$, we have for fixed $R$:
\begin{align*}
\E_{Q \in \csteifel{d}{b}} \charp{RAR^*+Q B Q^{*}}{x} 
	&=  \E_{Q \in \csteifel{d}{b}} \E_{U\in \csteifel{d}{d}} \charp{RAR^*+UQ B Q^{*}U^*}{x} \quad\textrm{by left invariance of $Q$} \\
	&= \E_{Q \in \csteifel{d}{b}} \charp{RAR^*}{x}\sqsum_{d} \charp{QBQ^*}{x}\quad\textrm{by \eqref{eqn:symmetricSum}}\\
	&=  \charp{RAR^*}{x}\sqsum_{d} \E_{Q \in \csteifel{d}{b}}\charp{QBQ^*}{x}\quad\textrm{by bilinearity of $\sqsum_{d}$}\\
	&= \charp{RAR^*}{x}\sqsum_{d} \frac{c!}{d!}D^{b-d}\charp{B}{x}\quad\textrm{by Lemma \ref{lem:restrict}.}
\end{align*}
	Averaging over $R$ and invoking bilinearity and Lemma \ref{lem:restrict} once more finishes the proof.
\end{proof}
Note that the definition \eqref{eqn:sqsumcd} is consistent with Lemma \ref{lem:reduce_degree}, e.g. if $b=d$ then the right hand side of \eqref{eqn:sqsumcd} is equal to
$p(x)\sqsum_a q(x)$.
As differentiation preserves real-rootedness, the generalized definition of $\sqsum_d$ preserves real-rootedness of polynomials of all degrees, and
bounds on their roots may be obtained from our results by reducing to the equal degree case by differentiating sufficiently many times.

While Lemma \ref{lem:restrict} can be used in conjunction with $\sqmult_d$ and
$\recsum_d$ just as easily due to their bilinearity, this does not correspond
to the appropriate degree-reduction operators (see Lemmas \ref{lem:multdegm1} and
\ref{lem:rectConvDiffDegs}) for those cases, so it does not lead to a
satisfying generalization of the definitions to higher degrees.
\subsection{Notation and Organization}
Let $\rrpoly{d}$ be the family of real rooted polynomials of degree exactly $d$ with positive leading coefficient,
  and let $\RRpoly$ be the union over $d$ of $\rrpoly{d}$.
Let $\rrpos{d}$ be the subset of these polynomials having only nonnegative roots.
We let $\RRpos $ be the union of $\rrpos{d}$ over all $d \geq 1$.
We also define $\rrneg{d}$ and $\RRneg$ to be the set of polynomials having only nonpositive roots.

For a function $f (x)$, we write the derivative as $\deriv f (x)$.
For a number $\alpha$, we let $\shiftoper{\alpha}$ be the operator that maps
  $f$ to $f - \alpha D f$.
That is, $\shiftoper{\alpha}$ is multiplication by $1 - \alpha D$.

\renewcommand{\mm}{m}
\renewcommand{\nn}{n}

\renewcommand{\aa}{d}
\renewcommand{\bb}{d}
\renewcommand{\cc}{d}
\renewcommand{\dd}{d}

\newcommand{\bin}{ \frac{1}{\binom{ \max \{ \mm, \nn \} }{k}}}
\newcommand{\expect}[2]{\mathbb{E}_{#1} \left\{\, #2 \,\right\}}
\newcommand{\mydel}[1]{\delta_{ \{ #1 \} } }
\newcommand{\minor}[3]{[#1]_{#2, #3}}

\section{Equivalence of Convolutions and $\E\chi$}\label{sec:formulas}

The goal of this section is to prove Theorems \ref{thm:symmetricAdditive},
\ref{thm:symmetricMult}, and \ref{thm:asymmetricAdditive} relating the three
convolutions to random matrices.  While we have so far only considered averages
over unitary matrices, it turns out that one can average over various other
collections of matrices and get the same formula.  In
Section~\ref{sec:minor-orthogonality}, we will define a property that we call
{\em minor-orthogonality} and then in Section~\ref{sec:formulas} we will show
that the coefficients we are interested can be computed using an average over
any collection of minor-orthogonal matrices.  Also in
Section~\ref{sec:minor-orthogonality}, we will show that the collection of $n
\times n$ signed permutation matrices (under a uniform distribution) is
minor-orthogonal, and then we will use this to show that the orthogonal
matrices (under the Haar measure) is minor-orthogonal.

There are some advantages to being able to express the convolutions as averages
over different collections of matrices; in particular, a formula that is easily
derived by replacing a unitary average by one over signed permutation matrices
will be used in the proof of Lemma~\ref{lem:legendre}.

\subsection{Minor-Orthogonality}\label{sec:minor-orthogonality} We will write
$[n]$ to denote the set $\{ 1, \dots, n \}$ and for a set $S$, we write
$\binom{S}{k}$ to denote the collection of subsets of $S$ that have exactly $k$
elements.  When our sets contain integers (which they always will), we will
consider the set to be ordered from smallest to largest.  Hence, for example,
if $S$ contains the elements $\{ 2, 5, 3 \}$, then we will write 
\[
S = \{ s_1, s_2, s_3 \} 
\quad
\text{where}
\quad
s_1 = 2, s_2 = 3, s_3 = 5.
\]
Now let $S = \{ s_1, \dots, s_k \} \in \binom{[n]}{k}$.
For a set $W \in \binom{[k]}{j}$ with $j \leq k$, we will write
\[
W(S) = \{ s_i : i \in W \}.
\]
Lastly, for a set of integers $S$, we will write 
\[
\| S \|_1 = \sum_{s \in S} s.
\]

\begin{example}
For $W = \{ 1, 3 \}$ and $S = \{ 2, 4, 5 \}$ we have
\[
W(S) = \{ 2, 5 \}
\AND
\|W\|_1 = 1 + 3 = 4
\AND
\|S\|_1 = 2 + 4 + 5 = 11.
\]
\end{example}

Let $\mm, \nn$ be positive integers.
Given an $\mm \times \nn$ matrix $A$ and sets $S \subseteq [\mm], T \subseteq [\nn]$ with $|S| = |T|$, we will write the {\em $(S, T)$-minor of $A$} as
\[
\minor{A}{S}{T} = \mydet{ \{ a_{i, j} \}_{i \in S, j \in T}}
\]
By definition, we will set $[A]_{\varnothing, \varnothing} = 1$.
A well-known theorem of Cauchy and Binet relates the minor of a product to the product of minors (\cite{horn2012matrix}):
\begin{theorem}\label{thm:CauchyBinet}
For integers $m, n, p, k$ and $m \times n$ matrix $A$ and $n \times p$ matrix $B$, we have
\begin{equation}\label{eq:mult}
\minor{AB}{S}{T} = \sum_{|U| \in \binom{[n]}{k}}
\minor{A}{S}{U} \minor{B}{U}{T}
\end{equation}
for any sets $S \in \binom{[m]}{k}$ and $T \in \binom{[p]}{k}$.
\end{theorem}

\begin{definition}
We will say that an $\mm \times \nn$ random matrix $R$ is {\em minor-orthogonal} if for all integers $k, \ell \leq \max \{ \mm, \nn \}$ and 
all sets $S, T, U, V$ with $|S| = |T| = k$ and $|U| = |V| = \ell$, we have
\[
\expect{R}{\minor{R}{S}{T}\minor{R^*}{U}{V}} 
= \bin \mydel{S = V}\mydel{T = U}.
\]
\end{definition}
Given a minor-orthogonal matrix $R$ it is easy to see from the definition that
\begin{enumerate} 
\item $R^*$ is minor orthogonal
\item any submatrix that preserves the largest dimension of $R$ is minor 
orthogonal
\end{enumerate}

\begin{lemma}\label{lem:rot}
If $R$ is minor-orthogonal and $Q$ is a fixed matrix for which $QQ^* = I$, then 
	$QR$ is minor-orthogonal. If $Q^*Q=I$ then $RQ$ is minor orthogonal.
\end{lemma}
\begin{proof}
For any sets $S, T$ with $|S| = |T| = k$, we have
\[
\minor{QR}{S}{T} = \sum_{|W| = k} \minor{Q}{S}{W} \minor{R}{W}{T}
\]
so for $|U| = |V| = \ell$, we have
\begin{align*}
\expect{R}{\minor{QR}{S}{T}\minor{(QR)^*}{U}{V}} 
&= 
\expect{R}{ 
\sum_{|W| = k}\sum_{|Z| = \ell} 
\minor{Q}{S}{W} \minor{R}{W}{T} 
\minor{R^*}{U}{Z} \minor{Q^*}{Z}{V} 
}
\\&= 
\sum_{|W| = k}\sum_{|Z| = \ell} \minor{Q}{S}{W} \minor{Q^*}{Z}{V} 
\bin \mydel{W = Z} \mydel{T = U}
\\&= 
\sum_{|W| = k} \bin \minor{Q}{S}{W} 
\minor{Q^*}{W}{V} \mydel{T = U}
\\&=
\bin
\mydel{S = V}\mydel{T = U},
\end{align*}
where the last line comes from the fact that $\minor{I}{S}{V} = \mydel{S = V}$.

The other case $RQ$ follows by repeating the argument with $R^*$ and noting that $RQ=(Q^*R^*)^*$.
\end{proof}

\begin{definition}
A \em{signed permutation matrix} is a matrix that can be written
  $E P$ where $E$ is a diagonal matrix with $\pm 1$ diagonal entries and
  $P$ is a permutation matrix.
\end{definition}

\begin{lemma}\label{lem:sps}
A uniformly random $\nn \times \nn$ signed permutation matrix is minor-orthogonal.
\end{lemma}
\begin{proof}
We can write a uniformly random signed permutation matrix $Q$ as $Q = E_\chi 
P_\pi$ where $P_\pi$ is a uniformly random permutation matrix and $E_\chi$ is a 
uniformly random diagonal matrix with $\chi\in\{ \pm 1 \}^\nn$ on the diagonal (and the 
two are independent).
Hence for $|S| = |T| = k$ and $|U| = |V| = \ell$, we have
\begin{align*}
\expect{Q}{\minor{Q}{S}{T}\minor{Q^*}{U}{V}} 
&=
\expect{\chi, \pi}{\minor{E_\chi P_\pi}{S}{T} \minor{ P_\pi^* E_\chi}{U}{V}} 
\\&= 
\sum_{|W| = k}\sum_{|Z| = \ell} \expect{\chi, \pi}{\minor{E_\chi}{S}{W} 
\minor{P_\pi}{W}{T} \minor{ P_\pi^* }{U}{Z} \minor{E_\chi}{Z}{V}}.
\\&= 
\expect{\chi, \pi}{\minor{E_\chi}{S}{S} \minor{P_\pi}{S}{T} \minor{ P_\pi^* 
}{U}{V} \minor{E_\chi}{V}{V}}
\\&
\expect{\chi}{ \prod_{i \in S} \chi_i \prod_{j \in V} \chi_j}
\expect{\pi}{\minor{P_\pi}{S}{T} \minor{ P_\pi^* }{U}{V}}.
\end{align*}
where the penultimate line uses the fact that a diagonal matrix $X$ satisfies 
$\minor{X}{A}{B} = 0$ whenever $A \neq B$.
Now the $\chi_i$ are uniformly distributed $\{ \pm 1 \}$ random variables, so 
\[
\expect{\chi}{ \prod_{i \in S} \chi_i \prod_{j \in V} \chi_j} = \mydel{S = V}
\]
and so we have
\begin{align*}
\expect{Q}{\minor{Q}{S}{T}\minor{Q^*}{U}{V}} 
&= 
\expect{\pi}{\minor{P_\pi}{S}{T} \minor{ P_\pi^* }{U}{V}} \mydel{S = V}
\\&=
\expect{\pi}{\minor{P_\pi}{S}{T} \minor{ P_\pi }{S}{U}} \mydel{S = V}
\end{align*}
Furthermore, $\minor{P_\pi}{S}{T} = 0$ except when $T = \pi(S)$, so in order 
for both $\minor{P_\pi}{S}{T}$ and $\minor{ P_\pi }{S}{U}$ to be nonzero 
simultaneously requires $U = T$.
In the case that $U = T = \pi(S)$, $\minor{P_\pi}{S}{T} = \pm 1$, and so we have
\begin{align*}
\expect{Q}{\minor{Q}{S}{T}\minor{Q^*}{U}{V}} 
&= 
\expect{\pi}{\minor{P_\pi}{S}{T}^2} \mydel{S = V} \mydel{T = U}
\\&= 
\expect{\pi}{\mydel{\pi(S) = T}} \mydel{S = V} \mydel{T = U}
\end{align*}

The probability that a permutation 
length $\nn$ maps a set $S$ to a set $T$ with $|S| = |T| = k$ is 
\[
\frac{k!(\nn-k)!}{\nn!} = \frac{1}{\binom{\nn}{k}}
\]
and so for $|S| = |T| = k$, we have
\[
\expect{Q}{\minor{Q}{S}{T}\minor{Q^*}{U}{V}}  
= \frac{1}{\binom{\nn}{k}}  \mydel{S = V} \mydel{T = U}
\]
as required. The case $m> n$ follows by considering $R^*$ instead of $R$.
\end{proof}

\begin{corollary}
An $m\times n$ Haar random matrix from the Stiefel manifold $\csteifel{m}{n}$ is minor-orthogonal.
\end{corollary}
\begin{proof}
Let $R$ be such a random matrix, and assume $m\le n$.
As a signed permutation matrix is unitary, 
  $RQ$ is also Haar distributed for any fixed signed permutation matrix 
  $Q$. By Lemma~\ref{lem:rot} it is also minor-orthogonal.
Hence
\[
\expect{R}{\minor{R}{S}{T}\minor{R^*}{U}{V}} 
=
\expect{R}{\minor{R Q}{S}{T}\minor{(RQ)^*}{U}{V}} 
\]
and so 
\[
\expect{R}{\minor{R}{S}{T}\minor{R^*}{U}{V}} 
=
\expect{R, Q}{\minor{R}{S}{T}\minor{R^*}{U}{V}} 
=
\expect{R, Q}{\minor{RQ}{S}{T}\minor{(RQ)^*}{U}{V}},
\]
where $Q$ is a uniform signed permutation independent from $R$.
By Lemma~\ref{lem:sps}, $Q$ is minor-orthogonal and so, for fixed $R$,  
Lemma~\ref{lem:rot} implies that $RQ$ is minor-orthogonal. 
So 
\[
\expect{R}{\expect{Q}{\minor{RQ}{S}{T}\minor{(RQ)^*}{U}{V}}}
= 
\expect{R}{ \frac{1}{\binom{n}{k}} \mydel{S = V}\mydel{T = U} }
=
\frac{1}{\binom{n}{k}} \mydel{S = V}\mydel{T = U}
\]
as required.
\end{proof}

\subsection{Formulas}\label{ssec:formulas}

We begin this section by mentioning a well-known formula for the determinants of a sum of matrices (see \cite{marcus1990determinants}):

\begin{theorem}\label{thm:LaplaceExpansion}
For integers $k \leq \nn$, $\nn \times \nn$ matrices $A, B$, 
  and sets $S, T \in \binom{[n]}{k}$, we have
\begin{equation}\label{eq:add}
\minor{A + B}{S}{T} = \sum_{i=0}^{k} \sum_{U, V \in \binom{[k]}{i}} 
(-1)^{\| U\|_1 + \|V \|_1} 
\minor{A}{U(S)}{V(T)} 
\minor{B}{\overline{U}(S)}{\overline{V}(T)},
\end{equation}
where $\overline{U} = [k] \setminus U$.
\end{theorem}

We denote the coefficient of $(-1)^{k} x^{d-k}$ of the characteristic polynomial of a $d$-dimensional
  matrix $A$ by $\sigma_{k} (A)$, which we recall is the $k$th elementary symmetric function of
  the eigenvalues of $A$.  
We will repeatedly use the fact that
\begin{equation}\label{eqn:sigmak}
\sigma_{k}(A) = \sum_{|S|=k}\minor{A}{S}{S}.
\end{equation}

\begin{lemma}\label{lem:RBR}
For integers $m \leq n$, let $B$ be an $n \times n$ matrix and let $R$ be a $m \times n$ minor-orthogonal matrix independent from $B$.
Then for all $|S| = |T| = k$, we have
\[
\expect{R}{
\minor{R B R^*}{S}{T}}
=
\mydel{S = T} \frac{\sigma_k(B)}{\binom{n}{k}}
\]
\end{lemma}
\begin{proof}
By Theorem~\ref{thm:CauchyBinet}, 
\begin{align*}
\expect{R}{\minor{R B R^*}{S}{T}}
&=
\sum_{X, Y \in \binom{[n]}{k}}
\expect{R}{\minor{R}{S}{X} \minor{B}{X}{Y} \minor{R^*}{Y}{T}}
\\&=
\sum_{X, Y \in \binom{[n]}{k}} \minor{B}{X}{Y}
\expect{R}{\minor{R}{S}{X} \minor{R}{T}{Y}}
\\&= \sum_{X, Y \in \binom{[n]}{k}} \minor{B}{X}{Y} 
\frac{1}{\binom{n}{k}} 
\mydel{X = Y} \mydel{S = T}
\\&= \sum_{X \in \binom{[n]}{k}} \minor{B}{X}{X} \frac{1}{\binom{n}{k}} \mydel{S = T}.
\end{align*}
\end{proof}

The above Lemma yields a quick proof of Lemma \ref{lem:restrict}.
\begin{proof}[Proof of Lemma \ref{lem:restrict}]
	Let $A$ be an $a\times a$ matrix and let $Q$ be Haar distributed on $\csteifel{d}{a}$. The $k^{th}$ coefficient of $\E \charp{QAQ^*}{x}$ is equal to
\begin{align*}
	\expect{Q}{\sigma_k(QAQ^*)} &= \sum_{|S|=k} \expect{Q}{ [QAQ^*]_{S,S}}\\
	&= \sum_{|S|=k}\frac{\sigma_k(A)}{\binom{a}{k}}\quad\textrm{by Lemma \ref{lem:RBR} and minor-orthogonality of $Q$}\\
	&= \frac{\binom{d}{k}}{\binom{a}{k}},
\end{align*}
	which is precisely the $k^{th}$ coefficient of $\frac{d!}{a!}D^{a-d}\charp{Q}{x}$.
\end{proof}

\subsubsection{Symmetric Additive and Multiplicative Convolutions}

Using Lemma~\ref{lem:RBR}, we can easily prove Theorems \ref{thm:symmetricAdditive} and \ref{thm:symmetricMult} by showing equality of each coefficient as per \eqref{eqn:sigmak}.

%formulas for the coefficients of the symmetric multiplicative convolution (Theorem~\ref{thm:symmetricMultCoeff}) and the symmetric additive convolution (Theorem~\ref{thm:symmetricAdditiveCoeff}).

\begin{theorem}[Implies Theorem \ref{thm:symmetricAdditive}]\label{thm:symmetricAdditiveCoeff}
Let $A$ and $B$ be $\dd \times \dd$ matrices, and let $R$ be a random $\dd \times \dd$ minor-orthogonal matrix.
Then 
\[
\expect{R}{\sigma_k(A + R B R^*)} = \sum_{i=0}^k \frac{\binom{\dd-i}{k-i}}{\binom{\dd}{k-i}}\sigma_i(A)\sigma_{k-i}(B).
\]
\end{theorem}
\begin{proof}
By \eqref{eqn:sigmak}, Theorem~\ref{thm:LaplaceExpansion}, and then Lemma~\ref{lem:RBR}, we have
\begin{align*}
\expect{R}{\sigma_k(A + R B R^*)} 
&= \sum_{S \in \binom{[\cc]}{k}} \expect{R}{\minor{A + R B R^*}{S}{S}} 
\\&= 
\sum_{S \in \binom{[\cc]}{k}}
\sum_{i=0}^{k} \sum_{U, V \in \binom{[k]}{i}} 
(-1)^{\| U\|_1 + \|V \|_1} 
\minor{A}{U(S)}{V(S)} 
\expect{R}{\minor{R B R^*}{\overline{U}(S)}{\overline{V}(S)}}
\\&=
\sum_{S \in \binom{[\cc]}{k}}
\sum_{i=0}^{k} \sum_{U, V \in \binom{[k]}{i}} 
(-1)^{\| U \|_1 + \|V \|_1} 
\minor{A}{U(S)}{V(S)} \mydel{\overline{U}(S) = \overline{V}(S)} \frac{\sigma_{k-i}(B)}{\binom{\dd}{k-i}}
\\&=
\sum_{i=0}^{k} 
\frac{\sigma_{k-i}(B)}{\binom{\dd}{k-i}}
\sum_{S \in \binom{[\cc]}{k}}
\sum_{U \in \binom{[k]}{i}} 
\minor{A}{U(S)}{U(S)} 
\end{align*}
where the last equality uses the fact that $\overline{U}(S) = \overline{V}(S)$ if and only if $U = V$.
Finally, 
\[
\sum_{S \in \binom{[\cc]}{k}}
\sum_{U \in \binom{[k]}{i}} 
\minor{A}{U(S)}{U(S)} 
=
\binom{\cc-i}{k-i} \sum_{V \in \binom{[\cc]}{i}} \minor{A}{V}{V}
\]
as $\binom{\cc-i}{k-i}$ is the number of times a set $V$ appears as $V = U(S)$ for some $S$ and some $U$.
That is, the number of ways we can add elements to a set of size $V$ to obtain a set of size $k$.
\end{proof}

Using Theorem~\ref{thm:symmetricAdditiveCoeff} and Lemma~\ref{lem:sps} we can derive another useful formula for the symmetric additive convolution, this time as a function of the {\em roots}.
It states that $p(x) \sqsum_d q(x)$ is the average of all polynomials you can form by adding the roots of $p$ and $q$ pairwise.

\begin{theorem}\label{thm:symmetricAdditiveRoots}
For $p(x) = \prod_{i=0}^d (x - a_i)$ and $q(x) = \prod_{i=0}^d (x - b_i)$ we have
\[
p(x) \sqsum_d q(x) 
= \frac{1}{d!} \sum_{\sigma} \prod_{i=1}^{d} (x - a_i - b_{\sigma(i)})
\]
where the sum is over permutations $\sigma$ of $[d]$.
\end{theorem}
\begin{proof}
Let $A$ be the diagonal matrix with diagonal elements $\{ a_i \}$ and let $B$ be the diagonal matrix with diagonal elements $\{ b_i \}$.
By Theorem~\ref{thm:symmetricAdditiveCoeff}, we have
\begin{align*}
p(x) \sqsum_d q(x) 
&=
\mydet{x I - A} \sqsum_{d} \mydet{xI - B}
\\& = \expect{R}{ \mydet{xI - A - R B R^*}}.
\end{align*}
where the expectation can be taken over any minor-orthogonal random matrix $R$.
By Lemma~\ref{lem:sps}, we can take $R$ to be a (uniformly random) signed permutation matrix.
Since $A$ and $B$ are diagonal, it is easy to compute (for each fixed value of $R$)
\[
\mydet{xI - A - R B R^*} = \prod_i (x - a_i - b_{\sigma(i)})
\]
where $\sigma$ is the permutation part of $R$ (all of the signs cancel).
Averaging over these gives the result.
\end{proof}

\begin{theorem}[Implies Theorem \ref{thm:symmetricMult}]\label{thm:symmetricMultCoeff}
Let $A$ and $B$ be $\dd \times \dd$ matrices, and let $R$ be an $\dd \times \dd$ minor-orthogonal matrix.  \dan{I removed phrases like "R is independent from A", as dodging a ghost like this only creates confusion.}
Then 
\[
\expect{R}{\sigma_k(A R B R^*)} = \frac{\sigma_k(A)\sigma_k(B)}{\binom{\dd}{k}}.
\]
\end{theorem}
\begin{proof}
By Theorem~\ref{thm:CauchyBinet} and then Lemma~\ref{lem:RBR}, we have
\begin{align*}
\expect{R}{\sigma_k(A R B R^*)} 
&= \sum_{S \in \binom{[\cc]}{k}} \expect{R}{\minor{A R B R^*}{S}{S}} 
\\&= 
\sum_{S, T \in \binom{[\cc]}{k}}
\minor{A}{S}{T} \expect{R}{\minor{R B R^*}{T}{S}}
\\&= 
\sum_{S, T \in \binom{[\cc]}{k}}
\minor{A}{S}{T}
\mydel{T = S} \frac{\sigma_k(B)}{\binom{\dd}{k}}
\\&= 
\frac{\sigma_k(A)\sigma_k(B)}{\binom{\dd}{k}}.
\end{align*}
\end{proof}

\subsubsection{Asymmetric Additive Convolution}

The proof of the asymmetric convolution is a bit more involved, due to the appearance of a second random matrix.

\begin{lemma}\label{lem:QandR}
Let $B$ be a $\bb \times \dd$ matrix and let $Q$ and $R$ be $\cc \times \dd$ independently random minor-orthogonal matrices.
Then
\[
\expect{Q, R}{\minor{Q B R^*}{S}{T} \minor{(Q B R^*)^*}{U}{V} }
= 
\frac{ \mydel{T = U}\mydel{S = V}}{\binom{\bb}{k}\binom{\dd}{k}} \sigma_k(BB^*)
\]
for any $|S| = |T| = k$ and $|U| = |V| = \ell$.
\end{lemma}
\begin{proof}
By Theorem~\ref{thm:CauchyBinet} we have 
\[
\minor{Q B R^*}{S}{T} \minor{(Q B R^*)^*}{U}{V}
= 
\sum_{|W|=|Z| = k}
\sum_{|W'|=|Z'| = \ell}
\minor{Q}{S}{W}\minor{B}{W}{Z}\minor{R^*}{Z}{T}\minor{R}{U}{W'}\minor{B^*}{W'}{Z'}\minor{Q^*}{Z'}{V}
\]
where
\[
\expect{Q}{\minor{Q}{S}{W}\minor{Q^*}{Z'}{V}} = \frac{1}{\binom{\bb}{k}}\mydel{S = V}\mydel{W = Z'}
\]
and
\[
\expect{R}{\minor{R^*}{Z}{T}\minor{R}{U}{W'}} = \frac{1}{\binom{\dd}{\ell}}\mydel{T = U}\mydel{Z = W'}.
\]
Hence
\begin{align*}
\expect{Q, R}{\minor{Q B R^*}{S}{T} \minor{(Q B R^*)^*}{U}{V}}
&= 
\sum_{|W|=|Z|= k}
\frac{1}{\binom{\bb}{k}}
\frac{1}{\binom{\dd}{k}} \minor{B}{W}{Z}\minor{B^*}{Z}{W}\mydel{T = U}\mydel{S = V}
\\&= 
\sum_{|W|=k}
\frac{1}{\binom{\bb}{k}}
\frac{1}{\binom{\dd}{k}} \minor{BB^*}{W}{W} \mydel{T = U}\mydel{S = V}
\\&= \frac{1}{\binom{\bb}{k}}
\frac{1}{\binom{\dd}{k}} \sigma_k(BB^*) \mydel{T = U}\mydel{S = V}
\end{align*}
\end{proof}

\begin{theorem}[Implies Theorem \ref{thm:asymmetricAdditive}]\label{thm:asymmetricAdditiveCoeff}
Let $A$ and $B$ be $\aa \times \cc$ matrices and let $Q$ and $R$ be $\cc \times \dd$ minor-orthogonal matrices that are independent from $A, B$ and each other.
Then
\[
\expect{Q, R}{\sigma_k((A + Q B R^*)(A + Q B R^*)^*)}
= 
\sum_i \frac{\binom{\aa-i}{k-i}\binom{\cc-i}{k-i}}{\binom{\bb}{k-i}\binom{\dd}{k-i}} \sigma_i(AA^*)\sigma_{k-i}(BB^*).
\]
\end{theorem}
\begin{proof}
Let $|U| = k$. 
Then by Theorem~\ref{thm:CauchyBinet} we have
\[
\minor{(A + Q B R^*)(A + Q B R^*)^*}{U}{U}
= 
\sum_{|V|=k} 
\minor{A + Q B R^*}{U}{V}
\minor{(A + Q B R^*)^*}{V}{U}.
\]
By Theorem~\ref{thm:LaplaceExpansion},
\[
\minor{A + Q B R^*}{U}{V}
=
\sum_{i=0}^{k} \sum_{|W|=|Z|=i} (-1)^{\|W\|_1 + \|Z\|_1}
\minor{A}{W(U)}{Z(V)}
\minor{Q B R^*}{\overline{W}(U)}{\overline{Z}(V)}
\]
and
\[
\minor{(A + Q B R^*)^*}{V}{U}
=
\sum_{i=0}^{k} \sum_{|W'|=|Z'|=i} (-1)^{\|W'\|_1 + \|Z'\|_1} 
\minor{A^*}{W'(V)}{Z'(U)}
\minor{(Q B R^*)^*}{\overline{W'}(V)}{\overline{Z'}(U)}.
\]
Applying Lemma~\ref{lem:QandR} then gives
\begin{align*}
\expect{Q, R}{
\minor{Q B R^*}{\overline{W}(U)}{\overline{Z}(V)}
\minor{(Q B R^*)^*}{\overline{W'}(V)}{\overline{Z'}(U)}
}
&=
\frac{\mydel{\overline{W'}(V) = \overline{Z}(V)}\mydel{\overline{W}(U) = \overline{Z'}(U)}}{\binom{\bb}{k-i}\binom{\dd}{k-i}}
\sigma_{k-i}(BB^*)
\\&= 
\frac{\mydel{W' = Z}\mydel{W = Z'}}{\binom{\bb}{k-i}\binom{\dd}{k-i}}
\sigma_{k-i}(BB^*)
\end{align*}
where again we use the fact that $A(X) = B(X)$ if and only if $A = B$.
Hence
\begin{align*}
&\expect{Q, R}{
\minor{(A + Q B R^*)(A + Q B R^*)^*}{U}{U}
}
\\&= 
\frac{1}{\binom{\bb}{k-i}}\frac{1}{\binom{\dd}{k-i}}\sigma_{k-i}(BB^*)
\sum_{|V| = k} \sum_{i=0}^{k} \sum_{|W|=|Z|=i}
\minor{A}{W(U)}{Z(V)} \minor{A^*}{Z(V)}{W(U)}.
\end{align*}
Similar to Theorem~\ref{thm:symmetricAdditiveCoeff}, we have 
\begin{align*}
\sum_{|U|=|V|=k} \sum_{|W|=|Z|=i}
\minor{A}{W(U)}{Z(V)} \minor{A^*}{Z(V)}{W(U)}
&= \theta_{i, k} \sum_{|S| = |T| = i} \minor{A}{S}{T} \minor{A^*}{T}{S}
\\&= \theta_{i, k} \sum_{|S|= i} \minor{AA^*}{S}{S}
\\&= \theta_{i, k} \sigma_i(AA^*),
\end{align*}
where  $\theta_{i, k}$ is the number of ways to complete $S$ to a set of size $k$ {\em and} complete $T$ to a set of size $k$.
Since $S \subseteq [\aa]$ and $T \subseteq [\cc]$, this is precisely $\binom{\aa-i}{k-i}\binom{\cc-i}{k-i}$, completing the proof.
\end{proof}

\section{Real Rootedness of the Asymmetric Additive Convolution}\label{sec:rr}

We will use the theory of stable polynomials to prove Theorem \ref{thm:asymAdditiveRR}.
For this theorem, we will require Hurwitz stable polynomials.
We recall that a  multivariate polynomial
   $p (z_{1}, \dots , z_{m}) \in \Reals{}[z_{1}, \dots , z_{m}]$
  is \textit{Hurwitz stable} if it
   is identically zero or if whenever the real part of $z_{i}$ is positive for all $i$,
  $p (z_{1}, \dots , z_{m}) \not = 0$.

Instead of proving Theorem \ref{thm:asymAdditiveRR} directly, we prove the following theorem
  from which it follows by substituting $-x$ for $x$.

\begin{theorem}\label{thm:asymNegRR}
Let
\[
  p (x) = \sum_{i=0}^{d} x^{d-i} a_{i}
\quad \text{and} \quad 
  q (x) = \sum_{i=0}^{d} x^{d-i} b_{i}
\]
be in $\rrneg{d}$.
Then,
\[
  r (x) 
 \defeq 
 \sum_{k=0}^{d} x^{d-k} \sum_{i+j = k}
 \left(\frac{(d-i)! (d-j)!}{d! (d-i-j)!} \right)^{2}   a_{i} b_{j}
\]
is also in $\rrneg{d}$.
\end{theorem}

We will use the following result to prove that a polynomial is in $\RRneg$.

\begin{lemma}\label{lem:rrneg}
Let $r (x)$ be a polynomial such that 
  $h (x,y) = r (xy)$ is Hurwitz stable.
Then, $r \in \RRneg$.
\end{lemma}
\begin{proof}
Let $\zeta $ be any root of $r$.
If $\zeta$ is neither zero or negative, then it has a square root with positive real part.
Setting both $x$ and $y$ to this square root would contradict the Hurwitz stability of $h$.
\end{proof}

We will prove that $r(x)$ is in $\RRneg$
  by constructing a Hurwitz stable polynomial and applying Lemma \ref{lem:rrneg}.
To this end, we need a few tools for constructing Hurwitz stable polynomials.

The first is elementary.
\begin{claim}\label{clm:pxy}
If $p (x) \in \RRneg$, then the polynomial $f (x,y) = p (xy)$ is Hurwitz stable.
\end{claim}
\begin{proof}
If both $x$ and $y$ have positive real part, then $xy$
  cannot be a nonpositive real, and thus cannot be a root of $p$.
\end{proof}

The second tool is the following result of Borcea and \Branden, which is a consequence of Corollary 5.9 of \cite{branden2007polynomials}. \dan{so, why do we cite this instead of the result?}
\begin{proposition}[Polarization]\label{pro:bbpolarization}
Let 
\[
p (x,y) = \sum_{i=0}^{d}\sum_{j=0}^{d} c_{i,j} x^{i} y^{j}
\]
be a Hurwitz stable polynomial.
For each integer $i$, let $\sigma_{i}^{x}$ be the $i$th elementary symmetric
  polynomial in the variables $x_{1}, \dots , x_{d}$, and let $\sigma_{i}^{y}$
  be the corresponding polynomial in $y_{1}, \dots , y_{d}$.
Then, the polynomial
\[
  P (x_{1}, \dots , x_{d}, y_{1}, \dots , y_{d})
\defeq 
\sum_{i=0}^{d}\sum_{j=0}^{d} c_{i,j} 
\frac{\sigma_{i}^{x} \sigma_{j}^{y}}{\binom{d}{i} \binom{d}{j}}
\]
is Hurwitz stable.
\end{proposition}
The polynomial $P$ is called the \textit{polarization} of $p$.
We remark that $P (x, \dots ,x, y, \dots , y) = p (x,y)$.

The last result we need is due to  Lieb and Sokal \cite{LiebSokal} (see also \cite[Theorem 8.4]{borceaBrandenLeeYang2}).
\begin{theorem}\label{thm:diffoper}
Let $P (z_{1}, \dots , z_{d})$ and $Q (z_{1}, \dots , z_{d})$ be Hurwitz stable polynomials.
Let $\deriv_{i}^{z}$ denote the derivative with respect to $z_{i}$.
Then,
\[
  Q (\deriv_{1}^{z}, \dots , \deriv_{d}^{z}) P (z_{1}, \dots , z_{d})
\]
is Hurwitz stable.
\end{theorem}

\begin{proof}[Proof of Theorem~\ref{thm:asymNegRR}]
Define $f (x,y) = p (xy)$ and $g (x,y) = (xy)^{d} q (1/xy)$.
Let $F (x_{1}, \dotsc , x_{d}, y_{1}, \dotsc , y_{d})$
  be the polarization of $f (x,y)$ in the variables
  $x_{1}, \dots , x_{d}, y_{1}, \dots , y_{d}$.
Let $G (x_{1}, \dots , x_{d}, y_{1}, \dots , y_{d})$ be the analogous polarization of $g (x,y)$.

Let $\sigma_{i}^{x}$ be the $i$th elementary symmetric function in
  $x_{1}, \dotsc , x_{d}$, and let $\delta_{i}^{x}$ be the $i$th
  elementary symmetric function in $D^{x}_{1}, \dotsc , D^{x}_{d}$.
Define $\sigma_{i}^{y}$ and $\delta_{i}^{y}$ analogously.

Then,
\[
  F (x_{1}, \dotsc , x_{d}, y_{1}, \dots , y_{d})
  = 
  \sum_{i=0}^{d}  a_{i} 
  \frac{\sigma_{d-i}^{x} \sigma_{d-i}^{y} }{\binom{d}{i}^{2}},
\]
and
\[
	G (D^{x}_{1}, \dotsc , D^{x}_{d}, D^{y}_{1}, \dotsc , D^{y}_{d})
  = 
  \sum_{i=0}^{d}  b_{i} 
  \frac{\delta_{i}^{x}  \delta_{i}^{y} }{\binom{d}{i}^{2}}.
\]
Define 
\[
H (x_{1}, \dotsc , x_{d}, y_{1}, \dots , y_{d})
  =
  G (D^{x}_{1}, \dotsc , D^{x}_{d}, D^{y}_{1}, \dotsc , D^{y}_{d})
  F (x_{1}, \dotsc , x_{d}, y_{1}, \dots , y_{d}).
\]
We know from Theorem \ref{thm:diffoper} that $H$ is Hurwitz stable.
Define
\[
  h (x,y) =   H (x, \dots, x, y, \dots , y).
\]
It is immediate that $h$ is Hurwitz stable too.
We will prove that $h (x,y) = r (xy)$,
  which by Lemma~\ref{lem:rrneg} implies that $r$ is in $\RRneg$.

It will be convenient to know that
\[
  \delta^{x}_{i} \sigma^{x}_{j}
=
\begin{cases}
\binom{d+i-j}{i} \sigma^{x}_{j-i}& \text{if $i \leq j$}           
\\
0  & \text{otherwise}.
\end{cases}
\]
We may now compute
\begin{align*}
H (x_{1},\dotsc x_{d}, y_{1}, \dots , y_{d})
& = 
  \sum_{i=0}^{d} 
   b_{i} 
  \frac{\delta_{i}^{x}  \delta_{i}^{y} }{\binom{d}{i}^{2}}
  \sum_{j=0}^{d}
   a_{j} 
  \frac{\sigma_{d-j}^{x} \sigma_{d-j}^{y} }{\binom{d}{j}^{2}}
\\
& = 
  \sum_{i=0}^{d}
  \sum_{j : i \leq d-j}
  \frac{ b_{i}}{\binom{d}{i}^{2}}
  \frac{ a_{j}}{\binom{d}{j}^{2}}
  \delta_{i}^{x}   \delta_{i}^{y} 
  \sigma_{d-j}^{x}   \sigma_{d-j}^{y} 
\\
& = 
  \sum_{i +j \leq d} 
  \frac{a_{j} b_{i}}{\binom{d}{i}^{2} \binom{d}{j}^{2}}
  \delta_{i}^{x} 
  \delta_{i}^{y} 
  \sigma_{d-j}^{x} 
  \sigma_{d-j}^{y} 
\\
& = 
  \sum_{i +j \leq d} 
  \frac{a_{j} b_{i}}{\binom{d}{i}^{2} \binom{d}{j}^{2}}
  \binom{d+i - (d-j) }{i}^{2}
  \sigma_{d-i-j}^{x}
  \sigma_{d-i-j}^{y}
\\
& = 
  \sum_{i +j \leq d} 
  a_{j} b_{i}
 \frac{\binom{i+j}{i}^{2}}{\binom{d}{i}^{2} \binom{d}{j}^{2}}
  \sigma_{d-i-j}^{x}
  \sigma_{d-i-j}^{y}
\\
& = 
  \sum_{k = 0}^{d} 
  \sum_{i + j = k}
  a_{j} b_{i}
  \frac{\binom{k}{i}^{2}}{\binom{d}{i}^{2} \binom{d}{j}^{2}}
  \sigma_{d-k}^{x}
  \sigma_{d-k}^{y}.
\end{align*}
So, 
\begin{align*}
  h (x, y)
& = 
  \sum_{k = 0}^{d} 
  \sum_{i + j = k}
  a_{j} b_{i}
 \left(  \frac{\binom{d}{k} \binom{k}{i}}{\binom{d}{i} \binom{d}{j}} \right)^{2}
  x^{d-k} y^{d-k}.
\\
& = 
  \sum_{k = 0}^{d} 
  \sum_{i + j = k}
  a_{j} b_{i}
 \left(   \frac{
  (d-i) ! (d-j) !
}{
 d! (d-i-j) ! 
}   \right)^{2}
  x^{d-k} y^{d-k}.
\end{align*}
So, $r (xy) = h (x,y)$ and therefore must have only nonpositive real roots.
\end{proof}

%%% Local Variables:
%%% mode: latex
%%% TeX-master: "finiteConvMaster"
%%% End:

\section{Transform bounds}\label{sec:bounds}

All of our transform bounds are proved using the following lemma.
It allows us to pinch together two of the roots of a polynomial without
  changing the value of the Cauchy transform at a particular point.
Through judicious use of this lemma, we are able to reduce statements about arbitrary
  polynomials to statements about polynomials with just one root.

\begin{lemma}[Pinching]\label{lem:pinch}
Let $\alpha > 0$, $d \geq 2$, and let 
  $p (x) \in \rrpoly{d}$ have at least two distinct roots.
Write $p (x) = \prod_{i=1}^{d} (x - \lambda_{i}) $
  where $\lambda_{1} \geq \lambda_{2} \geq \dots \geq \lambda_{d}$
  and $\lambda_{1} > \lambda_{k}$ for some $k$.
Then there exist real $\mu$ and $\rho$ so that
  $p (x) = \phat (x) + \ptil (x)$,
  where
\[
  \ptil (x) = (x- \mu)^{2} \prod_{i \not \in \setof{1,k}} (x - \lambda_{i}) \in \rrpoly{d}
\quad \text{and} \quad 
  \phat (x) = (x- \rho ) \prod_{i \not \in \setof{1,k}} (x - \lambda_{i}) \in \rrpoly{d-1},
\]
and 
\begin{itemize}
\item [a.]  $\maxroot{\shiftoper{\alpha} \ptil } = \maxroot{\shiftoper{\alpha}  \phat } = \maxroot{\shiftoper{\alpha } p}$,
\item [b.]  $\lambda_{1} > \mu > \lambda_{k}$, and
\item [c.]  $\rho > \lambda_{1}$.
In particular, if $d \geq 3$ then $\phat$ has at least two distinct roots.
\end{itemize}
\end{lemma}
\begin{proof}
Let $t = \maxroot{\shiftoper{\alpha } p}$ and set
\[
  \mu = t - \frac{2}{1/(t - \lambda_{1}) + 1/(t - \lambda_{k})}.
\]
We have chosen $\mu$ so that
\[
  \frac{2}{t - \mu} = \frac{1}{t - \lambda_{1}} + \frac{1}{t - \lambda_{k}}, 
\]
which implies
\[
  \frac{\deriv \ptil (t)}{\ptil  (t)} =
  \frac{\deriv p (t)}{p (t)} = \frac{1}{\alpha},
\]
and thus $\maxroot{\shiftoper{\alpha} \ptil } = t$.
Our choice of $\mu$ also guarantees that 
   $t - \mu$ is the harmonic average of $t - \lambda_{1}$  and $t - \lambda_{k}$.
Thus, $\mu$ must lie strictly between $\lambda_{1}$ and $\lambda_{k}$, which implies part $b$.
As the harmonic mean of distinct numbers is less than their average,
  $t - \mu < (1/2) (2t - (\lambda_{1} + \lambda_{k}))$,
  which implies that 
\begin{equation}\label{eqn:pinchMuAvg}
\mu > (\lambda_{1} + \lambda_{k})/2.
\end{equation}
We have
\[
\phat (x) = p (x) - \ptil (x)
= 
 \left(  (2 \mu - (\lambda_{1} + \lambda_{k} )) x - (\mu^{2} - \lambda_{1} \lambda_{k}) \right)
  \prod_{i \not \in \setof{1,k}} (x-\lambda_{i}).
\]
This and inequality \eqref{eqn:pinchMuAvg} imply that $\phat (x) \in \rrpoly{d-1}$.
As $\shiftoper{\alpha}$ is linear, we also have $(\shiftoper{\alpha} \phat) (t) = 0$.
To finish the proof of part $a$, we need to show that the maximum root of
  $\shiftoper{\alpha } \phat$ is less than $t$.
The one root of $\phat$ that is not a root of $p$ is
\[
 \rho  \defeq  \frac{\mu^{2} - \lambda_{1} \lambda_{k}}{2 \mu - (\lambda_{1} + \lambda_{k} )}.
\]
To see that $t > \rho $, compute
\[
  t- \rho  = \frac{(\lambda_{1} - \mu ) (\mu -\lambda_{k})}{2 \mu - \lambda_{1} - \lambda_{k}},
\]
which we know is greater than $0$ because of \eqref{eqn:pinchMuAvg} and the fact that 
  $\mu$ is between $\lambda_{1}$ and $\lambda_{k}$.
This completes the proof of part $a$.

To prove part $c$, note that $2 \mu - (\lambda_{1} + \lambda_{k}) > 0$, and
\[
  (2 \mu - (\lambda_{1} + \lambda_{k})) (\rho - \lambda_{1})
=
  \mu^{2} - 2 \lambda_{1} \mu  + \lambda_{1}^{2} = (\mu -\lambda_{1})^{2} > 0.
\]
\end{proof}

The following lemma provides one of the facts we exploit about the decomposition $p(x) = \ptil (x) + \phat (x)$

\begin{lemma}\label{lem:max}
  Let $f, g, h$ be real rooted polynomials with positive leading coefficients such that $f = g + h$.
  Then
  \begin{equation}\label{eq:firstmax}
    \maxroot{f} \leq \max \left\{ \maxroot{g}, \maxroot{h} \right\}
  \end{equation}
  with equality if and only if 
  \begin{equation}\label{eq:secondmax}
     \maxroot{f} = \maxroot{g} = \maxroot{h}.
  \end{equation}
\end{lemma}
\begin{proof}
  Note that equality in (\ref{eq:secondmax}) clearly implies equality in (\ref{eq:firstmax}).
  Now, assume by way of contradiction that \eqref{eq:firstmax} is false, and let $x =\maxroot{f}$. 
  Since $g$ and $h$ have positive leading terms, $g(x), h(x) > 0$.
  Thus, $f(x) = g(x) + h(x) > 0$, a contradiction .
\end{proof}

\subsection{Symmetric additive convolution}\label{sec:sqConvBound}

\newcommand{\phata}{\shiftoper{\alpha} \phat}
\newcommand{\ptila}{\shiftoper{\alpha} \ptil}
\newcommand{\pa}{\shiftoper{\alpha} p}
\newcommand{\qa}{\shiftoper{\alpha} q}

\newcommand{\pqa}{\shiftoper{\alpha} (p \sqsum_{d} q)}
\newcommand{\pqta}{\shiftoper{\alpha} (\ptil  \sqsum_{d} q)}
\newcommand{\pqha}{\shiftoper{\alpha} (\phat  \sqsum_{d} q)}

We now prove the upper bound on the $R$-transform of $p \sqsum_d q$ stated in the introduction.

\begin{theorem}\label{thm:sqsumTrans}
For $p,q \in \rrpoly{d}$  and $\alpha >0$,
  \[
    \maxroot{\shiftoper{\alpha} (p \sqsum_{d} q)} + d \alpha
    \leq 
    \maxroot{\shiftoper{\alpha} p}
    +
    \maxroot{\shiftoper{\alpha} q},
  \]
with equality only if $p(x)$ or $q(x)$ has the form $(x-\lambda)^{d}$.
\end{theorem}

Theorem \ref{thm:symmetricAdditive} and Lemma \ref{lem:reduce_degree}
  tell us that if $q (x) \in \rrpoly{d}$ and $p (x) = x^{d-1}$,
  then $p (x) \sqsum_{d-1} q (x) = (1/d) D q (x)$.
As $\maxroot{U_{\alpha} x^{d-1}} = (d-1) \alpha$,
  the following lemma may be viewed as a restriction of Theorem~\ref{thm:sqsumTrans}
  to the case that $q (x) = x^{d-1}$.

\begin{lemma}\label{lem:derivP}
For $\alpha \geq 0$, $d \geq 2$, and $p \in \rrpoly{d}$,
\begin{equation}\label{eqn:derivP}
\maxroot{\shiftoper{\alpha }  D p} \leq \maxroot{\shiftoper{\alpha} p} - \alpha.
\end{equation}
with equality if and only if $p = (x - \lambda)^d$ for some real number $\lambda$.
\end{lemma}
\begin{proof}
If $p = (x-\lambda)^{d}$, then $\maxroot{\shiftoper{\alpha} p} = \lambda + d\alpha$
  and $\maxroot{\shiftoper{\alpha} \deriv p} = \lambda + (d-1)\alpha$, giving equality in 
  \eqref{eqn:derivP}.

We now prove the rest of the lemma by induction on $d$, with $d = 2$ being the base case.
To establish the lemma in the case that $d = 2$ and $p$ has roots $\lambda_{1} > \lambda_{2}$,
  note that 
\[ 
  r \defeq \maxroot{\shiftoper{\alpha} \deriv p} + \alpha = 2 \alpha + (\lambda_{1} + \lambda_{2})/2
\]
and
\[
  \shiftoper{\alpha} p(x) = x^2 - (\lambda_{1} + \lambda_{2} + 2 \alpha) x + \alpha (\lambda_{1} + \lambda_{2}) 
    + \lambda_{1} \lambda_{2}.
\]
As this polynomial has a positive leading term, the fact that $\maxroot{ \shiftoper{\alpha} p } > r$
  follows from
  \[
    \shiftoper{\alpha} p(r) = - (\lambda_{1} - \lambda_{2})^2 / 4 < 0.
  \]

For a real rooted polynomial $p$, define
\[
  \phi (p) = \maxroot{\shiftoper{\alpha} p} - \alpha - \maxroot{\shiftoper{\alpha } \deriv p} .
\]
We will prove by induction on $d$ that $\phi(p) > 0$ for all polynomials $p \in \rrpoly{d}$ that have more than one root.
  
Assume by way of contradiction that there exists a monic (without loss of generality, since $\phi$ is independent of scaling) polynomial $p \in \rrpoly{d}$ with at least two distinct roots
  for which $\phi(p) \leq 0$.
Let $[-R, R]$ be an interval containing all of the roots of $p$,
  and define $\rrpoly{d}[-R, R]$ to be all monic polynomials in $\rrpoly{d}$ with all roots in this interval.
Since $[-R,R]^{d}$ is a compact set and $\phi$ is a continuous function of the roots of $p$, there is a monic polynomial $p_0 \in \rrpoly{d}[-R, R]$ at which $\phi$ obtains its minimum.
Let $p_0$ be such a polynomial, so $\phi(p_0) \leq \phi(p) \leq 0$. 
We may assume that $p_0$ has at least two distinct roots, 
  because it is true if $\phi (p_0) < 0$ whereas if $\phi (p_0) = 0$ we may assume $p_0 = p$. 
  
Lemma~\ref{lem:pinch} implies that there exist polynomials $\phat$ and $\ptil$ with $p_0 = \phat + \ptil$ such that 
  \begin{enumerate}
    \item $\phat$ and $\ptil$ have positive first coefficients
    \item $\phat$ has degree $d-1$, and if $d \geq 3$ it has at least two distinct roots.
    \item $\ptil \in \rrpoly{d}[-R, R]$.
    \item $\maxroot{\shiftoper{\alpha} \ptil } = \maxroot{\shiftoper{\alpha}  \phat } = \maxroot{\shiftoper{\alpha } p_0}$
  \end{enumerate}
By linearity we have 
  \[
    {\shiftoper{\alpha} \deriv \ptil } + {\shiftoper{\alpha}   \deriv \phat } = {\shiftoper{\alpha } \deriv p_0}.
  \]
If $\maxroot{ \shiftoper{\alpha}   \deriv p_0 } \leq \maxroot{ \shiftoper{\alpha}   \deriv \phat }$, 
then
\begin{align*}
 \phi (p_0)  
  & = \maxroot{\shiftoper{\alpha} p_0} - \alpha - \maxroot{\shiftoper{\alpha } \deriv p_0} 
  \\
  & \geq \maxroot{\shiftoper{\alpha} \phat} - \alpha - \maxroot{\shiftoper{\alpha } \deriv \phat} 
  \\  
  & = \phi(\phat) 
  \\
  & > 0,
 \end{align*}
by the inductive hypothesis, as $\phat$ has degree $d-1$.
As this would contradict our assumption that $\phi (p) \leq 0$, we may assume
 $\maxroot{ \shiftoper{\alpha}   \deriv p_0 } > \maxroot{ \shiftoper{\alpha}   \deriv \phat }$
  and apply Lemma~\ref{lem:max} to conclude
  $\maxroot{ \shiftoper{\alpha} \deriv p_0 } < \maxroot{ \shiftoper{\alpha}   \deriv \ptil }$.
This implies
\begin{align*}
  \phi (p_0)  
   & = \maxroot{\shiftoper{\alpha} p_0} - \alpha - \maxroot{\shiftoper{\alpha } \deriv p_0} 
   \\
   & > \maxroot{\shiftoper{\alpha} \phat} - \alpha - \maxroot{\shiftoper{\alpha } \deriv \ptil} 
   \\  
   & = \phi(\ptil) ,
\end{align*}
contradicting the minimality of $p_0$.
Thus, we may conclude that $\phi (p) > 0$ for all $p \in \rrpoly{d}$ with at least two roots.

\end{proof}

%Most of the work in the proof of Theorem~\ref{thm:sqsumTrans} is devoted to the case in which
%  $p$ and $q$ have the same degree.
%In our proof, we will find it useful to define $p \sqsum_d q$ when 
%  $q$ has degree $d$ and the degree of $p$ is lower.
%In this case, we use the expression on the right-hand side of \eqref{eqn:symmetricSum2}, which is equivalent to \eqref{eqn:symmetricSum_p0}.
%When their degrees pf $p$ and $q$ are different, we use the following lemma to simplify their symmetric additive
%  convolution.

%\begin{lemma}\label{lem:reduce_degree}
%For $p \in \rrpoly{c}$ and $q \in \rrpoly{d}$ with $c < d$,
%\[
%p \sqsum_{c} q =  \frac{1}{d} p \sqsum_{c} (\deriv q)
%\quad \text{and} \quad
%p \sqsum_{d} q =  \frac{1}{d} p \sqsum_{d-1} (\deriv q).
%\]
%\end{lemma}

\begin{lemma}\label{lem:sqSumOne}
For $\alpha \geq 0$, $q = (x- \lambda)^{d}$ for some real $\lambda$ and $p \in \rrpoly{d}$,
\[
  \maxroot{\shiftoper{\alpha} (p \sqsum_{d} q)}
 = 
  \maxroot{\shiftoper{\alpha} p} + 
  \maxroot{\shiftoper{\alpha} q} - \alpha d.
\]
\end{lemma}
\begin{proof}
We can prove this either by manipulating the identity in \eqref{eqn:symmetricSum} and those following it,
 or by going through the defintion \eqref{def:sqsum}.
To pursue the latter route, let $A$ be a Hermitian matrix whose characteristic polynomials is $p$ and let $B = \lambda I$.
We then have
\[
  p (x) \sqsum_{d} q (x) = \expec{Q} \charp{A + \lambda I}{x}
   = p(x - \lambda).
\]
Thus, 
\[
\maxroot{\shiftoper{\alpha} (p \sqsum_{d} q)} 
 = \lambda + \maxroot{\shiftoper{\alpha} p} .
\]
On the other hand, $\maxroot{\shiftoper{\alpha} q} = \lambda + \alpha d$.
\end{proof}

\begin{lemma}\label{lem:deriv_oneroot}
If $p \in \rrpoly{d}$ for $d \geq 3$ and $D p$ has just one root,
  then $p$ has just one root.
\end{lemma}
\begin{proof}
If $D p = (x - \lambda)^{d-1}$, then $p$ can be written in the form
  $(x - \lambda)^{d} + c$ for some constant $c$.
If $d \geq 3$ and $c$ were a constant other than zero, then this polynomial
  would have at least two complex roots.
\end{proof}

Our proof of Theorem~\ref{thm:sqsumTrans} will be very similar to our proof
  of Lemma \ref{lem:derivP}.

\begin{proof}[Proof of Theorem~\ref{thm:sqsumTrans}]
Lemma \ref{lem:sqSumOne} proves the theorem in the case that
  either $p$ or $q$ can be written in the form $(x - \lambda)^d$.
So, we will prove that if neither $p$ nor $q$ is of the form
$(x - \lambda)^d$ then
\begin{equation}\label{eqn:sqsumTrans2}
    \maxroot{\shiftoper{\alpha} (p \sqsum_{d} q)} + d \alpha
    <
    \maxroot{\shiftoper{\alpha} p}
    +
    \maxroot{\shiftoper{\alpha} q}.
\end{equation}
We will prove this by induction on $d$, the maximum degree of $p$ and $q$.
The base case $d=1$ is handled by Lemma \ref{lem:sqSumOne}.
Assume $d\ge 2$ and fix a polynomial $q\in\rrpoly{d}$ with at least two roots.
For any polynomial $p$ in $\rrpoly{d}$, define
  \[
    \phi (p) = \maxroot{\shiftoper{\alpha} p} + \maxroot{\shiftoper{\alpha} q} - d \alpha - 
    \maxroot{\shiftoper{\alpha } (p \sqsum_{d} q)} .
  \]
As before, assume for contradiction that there exists a monic polynomial $p$ with at least two roots for which $\phi(p) \leq 0$.
Let $[-R, R]$ be an interval 
  containing the roots of $p$ and let $p_0$ minimize $\phi$ over all monic polynomials 
  whose roots are contained in this interval.
We may assume that $p_0$ has at least two roots because Lemma \ref{lem:sqSumOne} says it must if $\phi(p_0) < 0$, 
  and otherwise we may take $p_0 = p$.

Thus, we can apply Lemma~\ref{lem:pinch} to $p_0$ to obtain 
 polynomials $\phat \in \rrpoly{d-1}$ and $\ptil \in \rrpoly{d}$ such that
 \begin{itemize}
  \item [a.] $p_0 = \phat + \ptil$, which by the linearity of $\sqsum_d$ and $\shiftoper{\alpha}$ implies 
  \[
    \shiftoper{\alpha} p_0 \sqsum_d q = \shiftoper{\alpha} \phat \sqsum_d q + \shiftoper{\alpha} \ptil \sqsum_d q;
    \]
  \item [b.] the roots of $\ptil$ lie inside $[-R, R]$, and so $\phi (\ptil) \geq \phi (p_0)$; and
  \item [c.]  $\maxroot{\shiftoper{\alpha} \ptil } = \maxroot{\shiftoper{\alpha}  \phat } = \maxroot{\shiftoper{\alpha } p_0}$.
 \end{itemize}
By Lemma \ref{lem:reduce_degree}

\[
  \maxroot{\shiftoper{\alpha} \phat \sqsum_d q} =  \maxroot{\shiftoper{\alpha} \phat \sqsum_{d-1} D q}.
\]
As the degree of $D q$ is less than $d$, and $Dq$ has at least two distinct roots by Lemma \ref{lem:deriv_oneroot} we may apply our inductive hypothesis to conclude that
\begin{align*}
  \maxroot{\shiftoper{\alpha} \phat \sqsum_{d-1} D q}
  & \leq
  \maxroot{\shiftoper{\alpha} \phat}
  +
  \maxroot{\shiftoper{\alpha} D q}
  - (d-1) \alpha  
 \\ & <
  \maxroot{\shiftoper{\alpha} \phat}
  +
  \maxroot{\shiftoper{\alpha} q}
	- d \alpha  & \text{by Lemma \ref{lem:derivP}}
  \\ & =
  \maxroot{\shiftoper{\alpha} p_0}
  +
  \maxroot{\shiftoper{\alpha} q}
	- d \alpha  & \text{by property (c)}
  \\ & \leq   \maxroot{\shiftoper{\alpha} p_0 \sqsum_{d} q},
\end{align*}
as $\phi (p_0) \leq 0$.
Thus, property $(a)$ above and Lemma \ref{lem:max} imply that
\[
  \maxroot{\shiftoper{\alpha} \ptil \sqsum_{d} q}
  > 
  \maxroot{\shiftoper{\alpha} p_0 \sqsum_{d} q}.
\]
As $\maxroot{\shiftoper{\alpha} \ptil} = \maxroot{\shiftoper{\alpha} p_0}$, 
  this implies $\phi (\ptil) < \phi (p_0)$, a contradiction.
Thus, \eqref{eqn:sqsumTrans2} holds when both polynomials
  have at least two roots.

%We finish the proof by considering the case in which $c < d$ and  
%  $q$ has at least two roots.
%In this case, we apply Lemma \ref{lem:reduce_degree} to show
%\[
%  \maxroot{\shiftoper{\alpha} p \sqsum_{c} q}
%  = 
%  \maxroot{\shiftoper{\alpha} p \sqsum_{c} q^{(d-c)}}.
%\]
%Lemma \ref{lem:deriv_oneroot} implies that $q^{(d-c)}$ has at least two roots,
%and so we may apply the result for equal degrees to show
%\begin{align*}
%  \maxroot{\shiftoper{\alpha} p \sqsum_{c} q^{(d-c)}}
% & <
%  \maxroot{\shiftoper{\alpha} p }
%  +
%  \maxroot{\shiftoper{\alpha}  q^{(d-c)}}  
%  - c \alpha
% \\
%  & <
%  \maxroot{\shiftoper{\alpha} p }
%  +
%  \maxroot{\shiftoper{\alpha}  q}
%  - d \alpha,
%\end{align*}
%by Lemma \ref{lem:derivP}.
\end{proof}

\subsection{Symmetric multiplicative  convolution}\label{sec:multTrans}

\begin{theorem}\label{thm:multTransBound}
	For $p, q\in \rrpos{d}$ having only nonnegative real roots and $w > 0$,
  \[
    \strans{p \sqmult_{d} q}{w} \leq \strans{p}{w} \strans{q}{w},
  \]
  with equality only when $p$ or $q$ has only one distinct root.
\end{theorem}

We begin by considering the case in which $p = (x-\lambda)^{d}$.
We then have that
\[
\mtrans{p}{z} = \sum_{j \geq 1} (\lambda /z)^{j}
= 
\frac{\lambda}{z - \lambda}.
\]
Thus,
\[
\invmtrans{p}{w} = \frac{1+w}{w} \lambda ,
\]
and
\[
\strans{p}{w} = \lambda .
\]

\begin{lemma}\label{lem:multTransOne}
If $\lambda > 0$, $p (x) = (x-\lambda)^{d}$ and $q (x) \in \rrpos{d}$, then for all $w \geq 0$ 
\[
\strans{p \sqmult_{d} q}{w}
=
\strans{p}{w}
\strans{q}{w}.
 \]
\end{lemma}
\begin{proof}
For $p (x) = (x-\lambda)^{d}$, one may use either the definition \eqref{def:sqmult} or Theorem \ref{thm:symmetricMult}
  to compute
\[
  p (x) \sqmult_{d} q (x) = \lambda^{d} q (x/\lambda).
\]
As, $\mtrans{q (x/\lambda)}{\lambda z} = \mtrans{q (x)}{z}$, 
\[
  \strans{p (x) \sqmult_{d} q (x)}{w}
 = \lambda \strans{q (x)}{w} .
\]
\end{proof}
%We will now prove that this is the only case in which equality holds in Theorem~\ref{thm:multTransBound}.

The finite multiplicative convolution of polynomials of different degrees
  may be computed by taking the \textit{polar derivative with respect to 0} of the polynomial of higher degree.
We recall that the polar derivative at 0 of a polynomial $p$ of degree $d$ is given by
  $d p - x \deriv p$ (see \cite[p. 44]{Marden}).

\begin{lemma}[Degree Reduction for $\sqmult_d$]\label{lem:multdegm1}
For $p (x) \in \rrpoly{d}$ and $q (x) \in \rrpoly{k}$, for $k < d$,
\[
  p (x) \sqmult_{d} q (x) = \frac{1}{d} (d p (x) - x \deriv  p (x) ) \sqmult_{d-1} q (x).
\]
\end{lemma}
\begin{proof}
	Follows from equation \eqref{eqn:sqmultcompact} by an elementary computation.
\end{proof}

Let $R$ be the operation on polynomials in $\rrpos{d}$ that maps $p (x)$ to $x^{d} p (1/x)$.
The polar derivative of a degree $d$ polynomial may be expressed in terms of $R$ by
\[
  d p - x \deriv p = R \deriv R p.
\]

\begin{claim}\label{clm:polarDeriv}
For $p = (x-\lambda)^{d}$,
\[
 x \deriv p -  d p = \lambda d (x- \lambda)^{d-1}.
\]
For $p \in \rrpos{d}$, $(x D - d)p \in \rrpos{d-1}$ and
\begin{equation}\label{eqn:polarDeriv}
	\maxroot{p} \geq \maxroot{x \deriv p - d p}.
\end{equation}
	In the special case of $p\in\rrpos{2}$ with distinct roots, strict inequality holds:
\begin{equation}\label{eqn:quadraticMultStrict}
	\maxroot{p} > \maxroot{x\deriv p - 2p}.
\end{equation}

\end{claim}
\begin{proof}
The first part is a simple calculation.
Inequality \eqref{eqn:polarDeriv} follows from the fact that $p \in \rrpos{d}$ implies that $R p \in \rrpos{d}$
   and the fact that the roots of $\deriv R p$ interlace those of $R p$.
To see that $(x D - d)p \in \rrpos{d-1}$, observe that its lead coefficient is positive.

The last claim follows by noting that $R p$ is quadratic polynomial with distinct roots and so $D R p$ strictly interlaces $R p$.
\end{proof}

As we did with the symmetric additive convolution, we relate the
  $\widetilde{\mathcal{M}}$-transform to the maximum root of a polynomial.
We have
\[
\mtrans{p}{z} = w
\quad \iff \quad 
\maxroot{\left(1 - \frac{x D}{d (w+1)} \right) p (x)} = 0.
\]
We therefore define the operator $\multoper{w}$ by
\[
  \multoper{w} p (x)= \left( 1 - \frac{x \deriv }{d (w+1)} \right) p (x),
\]
which gives
\[
\mtrans{p}{z} = w
\quad \iff \quad 
\maxroot{\multoper{w} p} = z.
\]
Note that the polar derivative is $d \multoper{0}$.

Our proof of Theorem \ref{thm:multTransBound} will also employ the following consequence
  of Lemma~\ref{lem:pinch}.

\begin{corollary}\label{cor:multPinch}
Let $w > 0$, $d \geq 2$, and let 
  $p (x) \in \rrpos{d}$ have at least two distinct roots.
Then there exist $\ptil \in \rrpos{d}$ and $\phat \in \rrpos{d-1}$
  so that 
  $p (x) = \phat (x) + \ptil (x)$,
  the largest root of $\ptil$ is at most the largest root of $p$, 
\[
\maxroot{\multoper{w} p}
=
\maxroot{\multoper{w} \ptil}
=
\maxroot{\multoper{w} \phat},
\]
and if $d \geq 3$ then $\phat$ has two distinct roots. 
\end{corollary}
\begin{proof}
To derive this from Lemma \ref{lem:pinch}, let $t = \maxroot{\multoper{w} p}$
  and set
\[
\alpha = \frac{t}{d (w+1)}.
\]
The polynomials $\ptil$ and $\phat$ constructed in Lemma \ref{lem:pinch} now
  satisfy 
\[
\maxroot{\shiftoper{\alpha} \phat} = 
\maxroot{\shiftoper{\alpha} \ptil} = 
t
=
\maxroot{\multoper{w} \ptil} = 
\maxroot{\multoper{w} \phat}, 
\]
as desired.
\end{proof}

\begin{proof}[Proof of Theorem~\ref{thm:multTransBound}]
We proceed by induction on $d$, the maximum degree of $p$ and $q$.
The theorem is true for $d=1$ by Lemma \ref{lem:multTransOne}.
As we have already shown that equality holds when one of $p$ or $q$ has just one root,
  we need to show that when both $p$ and $q$ have at least two distinct roots:
\[
\maxroot{  \multoper{w} p \sqmult_{d} q}
<
\frac{w}{1+w}
\maxroot{  \multoper{w} p }
\maxroot{  \multoper{w} q }.
\]
Fix $q\in\rrpos{d}$ with at least two distinct roots, and for $p \in \RRpos$ define:
\[
 \phi (p) \defeq \frac{w}{1+w}
\maxroot{  \multoper{w} p }
\maxroot{  \multoper{w} q }
-
\maxroot{  \multoper{w} p \sqmult_{d} q}.
\]
As before, we assume (for contradiction) that there exists a monic $p$ with two distinct roots and $\phi(p) \le 0$. Choose an interval $[0, R]$ containing the roots of $p$, and let $p_0$ minimize $\phi$ over all degree $d$ monic polynomials with roots in this interval.
	Observe that we can choose $p_0$ having two distinct roots: if $\phi(p_0)<0$ this is implied by Lemma \ref{lem:multTransOne} and
	if $\phi(p_0)=0$ we can take $p=p_0$.
	Thus we may apply Corollary \ref{cor:multPinch} to obtain polynomials $\ptil\in\rrpos{d}$ and $\phat\in\rrpos{d-1}$
	with  $p_0 = \ptil + \phat$,
	\begin{equation}
		\label{eqn:allequalMult}\maxroot{\multoper{w} p_0} = \maxroot{\multoper{w} \ptil} = \maxroot{\multoper{w} \phat},
	\end{equation}
and $\maxroot{\ptil}\le \maxroot{p_0}$.  By Lemma~\ref{lem:max}, we have
\[
  \maxroot{\multoper{w} (p_0 \sqmult_{d} q)} \leq \max \left\{ \maxroot{\multoper{w} (\phat \sqmult_{d} q)}, \maxroot{\multoper{w} (\ptil \sqmult_{d} q)} \right\},
\]
with equality only if all three are equal.
	However, noting that $\phat$ and $(xD-d)q= -RDRq$ have two distinct roots whenever $d\ge 3$:
\begin{align*}
  \maxroot{\multoper{w}   (\phat \sqmult_{d} q)}
& =
	\maxroot{\multoper{w} \phat \sqmult_{d-1}   ((x \deriv -d)q)}
\quad \text{by Lemma~\ref{lem:multdegm1}}
\\
& \leq 
\frac{w}{1+w}
	\maxroot{  \multoper{w} \phat  }
	\maxroot{  \multoper{w}  (x \deriv   - d) q  ) }
\quad \text{by induction, strict for $d\ge 3$}
\\
& \le 
\frac{w}{1+w}
	\maxroot{  \multoper{w} \phat  }
	\maxroot{  \multoper{w}  q  ) } \quad \text{by Claim~\ref{clm:polarDeriv}, strict for $d=2$}
\\
& =
\frac{w}{1+w}
	\maxroot{  \multoper{w} p_0  }
	\maxroot{  \multoper{w}  q  ) } \quad \text{by \eqref{eqn:allequalMult}}
\\
&\le  \maxroot{\multoper{w}   (p_0 \sqmult_{d} q)},
\end{align*}
since $\phi(p_0)\le 0$.
Since at least one of the inequalities above is strict for all $d\ge 2$, we
	must have $ \maxroot{\multoper{w}   (p_0 \sqmult_{d}
	q)}<\maxroot{\multoper{w}   (\ptil \sqmult_{d} q)}$, which implies
	$\phi(\ptil)<\phi(p_0)$, contradicting the minimality of $p_0$.
\end{proof}
 
% We define
% \[
%   \beta  = \frac{w}{1+w}
% \maxroot{  \multoper{w} p }
% \maxroot{  \multoper{w} q }
% \quad \text{and} \quad 
% \beta ' = \maxroot{  \multoper{w} (p \sqmult_{d} q)},
% \]
% and assume by way of contradiction that $\beta ' > \beta$.
% By induction, we know that 
%   $\maxroot{\multoper{w} (\phat \ \sqmult_{d} q)} \leq \beta$, and so
%   $\multoper{w} (\phat \ \sqmult_{d} q) (\beta ') > 0$.
% This implies that   $\multoper{w} (\ptil \ \sqmult_{d} q) (\beta ') < 0$,
%   and thus   
%   $\maxroot{\multoper{w} (\ptil \ \sqmult_{d} q)} > \beta'$.
% As all of the roots of $\ptil$ are less than $R$, this contradicts the
%   assumption that $p$ minimizes $\phi (p)$.

\subsection{Asymmetric additive convolution}\label{sec:asymTrans}

The following is a restatement of Theorem \ref{thm:recsumTrans}.

\begin{theorem}\label{thm:rectConvBound}
Let $p (x)$ and $q (x)$ be in $\rrpos{d}$ for $d \geq 1$.
Then for all $\alpha \geq 0$,
\[
\maxroot{\shiftoper{\alpha} \subsquare (p \recsum_{d} q)}
\leq 
\maxroot{\shiftoper{\alpha} \subsquare p} + \maxroot{\shiftoper{\alpha} \subsquare q} - 2 \alpha d,
\]
with equality if and only if $p$ or $q$ equals $x^{d}$.
\end{theorem}

We remark that if $q (x) = x^{d}$, then $p \recsum_{d} q = p$,
  and
\[
\shiftoper{\alpha} \subsquare q = \shiftoper{\alpha} x^{2d} = x^{2d-1} (x - 2 d \alpha),
\]
so
\[
\maxroot{\shiftoper{\alpha} \subsquare q} = 2 \alpha d.
\]
This is why Theorem \ref{thm:rectConvBound} holds with equality when $q (x) = x^{d}$.

The following lemma tells us that it suffices 
  to prove Theorem~\ref{thm:rectConvBound} in the case that $\alpha = 1$.

\begin{lemma}\label{lem:alpha1}
For a real-rooted polynomial $p (x)$,
\[
\maxroot{\shiftoper{\alpha} p (x)}
 = 
 \frac{1}{\alpha} \maxroot{\shiftoper{1} p (\alpha x)}.
\]
\end{lemma}
\begin{proof}
Let $q (x) = p (\alpha x)$, so
\[
  \shiftoper{1} q (x) = p (\alpha x) - \alpha p' (\alpha x).
\]
Let 
\[
w = \smax{\alpha}{p (x)} = \maxroot{\shiftoper{\alpha } p}
\quad \iff \quad p (w) - \alpha p' (w) = 0.
\]
Then,
\[
  (\shiftoper{1} q) (w/\alpha) = p (w) - \alpha p' (w) = 0.
\]
\end{proof}

Our proof of Theorem~\ref{thm:rectConvBound} will use the following lemma to pinch together roots
  of $p$ to reduce the analysis to a few special cases.

\begin{corollary}\label{cor:recPinch}
Let $\alpha > 0$, $d \geq 2$, and let 
  $p (x) \in \rrpos{d}$ have at least two distinct roots.
Then there exist $\ptil \in \rrpos{d}$ and $\phat \in \rrpos{d-1}$
  so that 
  $p (x) = \phat (x) + \ptil (x)$,
  the largest root of $\ptil$ is at most the largest root of $p$,
  $\phat$ has a root larger than $0$, and
\begin{equation}\label{eqn:recPinch}
\maxroot{\shiftoper{\alpha} \subsquare \ptil }
 = \maxroot{\shiftoper{\alpha} \subsquare  \phat }
 = \maxroot{\shiftoper{\alpha} \subsquare  p }
\end{equation}
\end{corollary}
\begin{proof}
Let $t = \maxroot{\shiftoper{\alpha } \subsquare p}$,
  so
\[
  \maxroot{(1 - 2 \alpha t D)p} = \sqrt{t}.
\]
Apply Lemma \ref{lem:pinch} with $2 \alpha t$ in the place of $\alpha$
  to construct the polynomials $\ptil$ and $\phat$.
They satisfy
\[
\maxroot{\shiftoper{\alpha} \phat} = 
\maxroot{\shiftoper{\alpha} \ptil} = 
\sqrt{t},
\]
which implies \eqref{eqn:recPinch}.
\end{proof}

We will build up to the proof of Theorem \ref{thm:rectConvBound} by first handling three
  special cases:

\begin{itemize}
\item When $p (x) = (x-\lambda)^{d}$ and $q (x) = x^{d-1}$.   That is, we consider $\dxd (x-\lambda)^{d}$
  (Lemma \ref{lem:p1q0}).

\item When $p (x) \in \rrpos{d}$  and $q (x) = x^{d-1}$.
  That is, we consider $\dxd p (x)$
  (Lemma \ref{lem:pxq0}).

\item When $p (x) = (x-\lambda )^{d}$ and $q (x) = (x-\mu )^{d}$
  (Lemma \ref{lem:p1q1}).
\end{itemize}

As with the other convolutions, we may compute the asymetric additive convolution of two
  polynomials by first applying an operation to the polynomial of higher degree.
In this case it is $\dxd$, also known as the ``Laguerre Derivitive''. \dan{cite something?}

\begin{lemma}[Degree Reduction for $\recsum_d$]\label{lem:rectConvDiffDegs}
Let $p \in \rrpos{d}$ and let $q \in \rrpos{k}$ for $k < d$.
Then,
\[
  p \recsum_{d} q = (1/d^{2}) (\dxd p) \recsum_{d-1} q.
\]
\end{lemma}
\begin{proof}
Follows from Theorem \ref{thm:asymmetricAdditive}.
\end{proof}

The following characterization of the Laguerre derivitive also follows
  from Theorem \ref{thm:asymmetricAdditive}.

\begin{claim}\label{clm:DxD}
If $q (x) = x^{d-1}$, then
\[
  p \recsum_{d} q = \deriv x \deriv p.
\]
\end{claim}

\begin{lemma}\label{lem:p1q0}
For $\alpha \geq  0$, $\lambda \geq 0$, $d \geq 2$ and $p (x) = (x-\lambda)^{d}$,
\[
\maxroot{\shiftoper{\alpha} \subsquare \dxd p}
\leq 
\maxroot{\shiftoper{\alpha} \subsquare p} - 2 \alpha ,
\]
with equality only if $\lambda = 0$.
\end{lemma}

\begin{proof}
By Lemma~\ref{lem:alpha1}, it suffices to consider the case of $\alpha = 1$.
As $\subsquare  p (x) = (x^{2} - \lambda)^{d}$,
\[
  \shiftoper{1 } \subsquare p (x) = (x^{2} - 2 \lambda  d - \lambda ) (x^{2} - \lambda)^{d-1}.
\]
So, the largest root of this polynomial is the largest root of 
\[
  r_{\lambda} (x) \defeq  (x^{2} - 2 \lambda d - \lambda ).
\]

We may also compute
\[
  \shiftoper{1 }  \subsquare \deriv x \deriv (x^{2}- \lambda)^{d}
= 
  q_{\lambda } (x) (x^{2} - \lambda)^{d-2},
\]
where
\[
  q_{\lambda} (x) \defeq d x^{4} - 2 d (d-1) x^{3}  - (d+1) \lambda x^{2} + 4 (d-1) \lambda x + \lambda^{2}.
\]

We now prove that
\[
 \maxroot{q_{\lambda }}  \leq \maxroot{r_{\lambda }} - 2,
\]
with equality only if $\lambda = 0$.

We first argue that $q_{\lambda} (x)$ is real rooted.
This follows from that fact that it is a factor of
  $\shiftoper{1} \subsquare \deriv x \deriv (x-\lambda)^{d}$.
For $\lambda \geq  0$
  all of the roots of $\deriv x \deriv (x-\lambda)^{d}$ are nonnegative, and so
  applying $\subsquare$ to it yields a polynomial with all real roots.

We now compute
\[
  \maxroot{r_{\lambda }} = d + \sqrt{d^{2} + \lambda}.
\]
Define
\[
  \mu_{\lambda} = d + \sqrt{d^{2} + \lambda} - 2.
\]
Elementary algebra gives
\[
  q_{\lambda } (\mu_{\lambda}) = 4 \lambda - 8 d \sqrt{d^{2} + \lambda} + 8 d^{2}
= 
  (2 d - 2 \sqrt{d^{2} + \lambda})^{2}.
\]
So, $q_{\lambda } (\mu_{\lambda}) \geq 0$,
  with equality only when $\lambda = 0$.
With just a little more work, we will show that $\mu_{\lambda}$ is an upper bound
  on the roots of $q_{\lambda }$ for all $\lambda$.

For $q_{\lambda }$ to have a root larger than $\mu_{\lambda}$, it would have to have two roots
  larger than $\mu_{\lambda}$.
When $\lambda = 0$, the polynomial $q_{\lambda }$ has one root at $\mu_{0}$ and a root at $0$
  with multiplicity $3$.
As $q_{\lambda }$ is real rooted for all $\lambda \geq 0$ and the roots of $q_{\lambda }$ are continuous functions
  of its coefficients, and thus of $\lambda$, we can conclude that for small $\lambda$ all but one of the roots
  of $q_{\lambda}$ must be near $0$.
Thus, for sufficiently small $\lambda$, $q_{\lambda}$ can have at most one root greater than $\mu_{\lambda}$, and so
  it must have none.
As the largest root of $q_{\lambda}$ and $\mu_{\lambda}$ are continuous function of $\lambda$, 
  $\maxroot{q_{\lambda}} > \mu_{\lambda}$ for all sufficiently  small $\lambda $.
As
  $q_{\lambda} (\mu_{\lambda}) > 0$ for all $\lambda  \geq 0$, we can conclude that
  $\maxroot{q_{\lambda}} > \mu_{\lambda}$ for all $\lambda \geq 0$.
\end{proof}

To see that Lemma~\ref{lem:p1q0} is equivalent Theorem~\ref{thm:rectConvBound} in the case of $q = x^{d-1}$,
  note that for $q (x) = x^{d-1}$,
\[
  \shiftoper{\alpha} \subsquare q (x) = \shiftoper{\alpha} q (x^{2}) = x^{2 (d-1)} - \alpha D x^{2 (d-1)}
  = x^{2d-3} (x - 2 (d-1) \alpha).
\]
The equivalence now follows from Claim~\ref{clm:DxD} and the fact that the 
  the largest root of this polynomial is $2 (d-1) \alpha$.

\begin{lemma}\label{lem:pxq0}
For $\alpha \geq  0$, $d \geq 2$ and $p \in \rrpos{d}$,
\[
\maxroot{\shiftoper{\alpha} \subsquare \dxd p}
\leq 
\maxroot{\shiftoper{\alpha} \subsquare p} - 2 \alpha,
\]
with equality only if $p (x) = x^{d}$.
\end{lemma}
\begin{proof}
For every $p \in \RRpos$, define
\[
  \phi (p) = 
\maxroot{\shiftoper{\alpha} \subsquare p}
- 
\maxroot{\shiftoper{\alpha} \subsquare \dxd p}
-
2 \alpha .
\]
We will show that $\phi (p) \geq 0$ for every polynomial $p \in  \RRpos$ of degree at least $2$,
   with equality only when $p = x^{d}$.

Our proof will be by induction on the degree of $p$.
Assume that there exists a polynomial $p \in \rrpos{d}$ with $\phi(p) < 0$, let $[0, R]$ be an interval containing the roots of $p$, and let $p_0$ minimize $\phi$ over polynomials with roots in that interval.
By Lemma~\ref{lem:p1q0}, $p_0$ must have at least 2 distinct roots, and so we can apply Corollary~\ref{cor:recPinch} to obtain
   polynomials $\phat$ and $\ptil$.
   
 Let $x = \maxroot{\shiftoper{\alpha} \subsquare \dxd  p}$.
  If $d=2$, then $\phat$ has degree $1$ and so $\shiftoper{\alpha} \subsquare \dxd \phat$ equals the lead coefficient of $\phat$, which implies
\begin{equation}\label{eqn:p1q0}
   \shiftoper{\alpha} \subsquare \dxd \phat (x) > 0.
\end{equation}
 For $d \geq 3$, we can then assume by induction that $\phi (\phat) > 0$, which then implies \eqref{eqn:p1q0} as well.
 Combining this with Lemma~\ref{lem:max}, we get that $\phi(p_0) > \phi(\ptil)$ which contradicts the minimality of $p_0$.

\end{proof}

In Section \ref{sec:cheby}, we establish the following special case of Theorem~\ref{thm:rectConvBound}.
\begin{lemma}\label{lem:p1q1}
For $\lambda, \mu > 0$, and $d \geq 1$, let $p (x) = (x - \lambda )^{d}$ and $q (x) = (x-\mu )^{d}$.
Then for all $\alpha \geq 0$,
\[
\maxroot{\shiftoper{\alpha} \subsquare (p \recsum_{d} q)}
<
\maxroot{\shiftoper{\alpha} \subsquare p} + \maxroot{\shiftoper{\alpha} \subsquare q} - 2 \alpha d.
\]
\end{lemma}

We now use Lemma~\ref{lem:p1q1} to prove Theorem~\ref{thm:rectConvBound} through 
  a variation of the pinching argument employed in the proof of Lemma~\ref{lem:p1q0}.

\begin{proof}[Proof of Theorem~\ref{thm:rectConvBound}]
We will prove this by induction on the maximum degree of $p$ and $q$, which we call $d$.
Our base case of $d=1$ is handled by Lemma~\ref{lem:p1q1}.

Assume, without loss of generality, that the degree of $p$ is at least the degree of $q$.
If the degree of $p$ is larger than the degree of $q$, then we may
  prove the hypothesis by
\begin{align*}
  \maxroot{\shiftoper{\alpha} \subsquare  (p \recsum_{d} q)}
& =
  \maxroot{\shiftoper{\alpha} \subsquare  ((\dxd  p) \recsum_{d-1} q)}
\quad \text{(by Lemma~\ref{lem:rectConvDiffDegs})}
\\
& \leq \maxroot{\shiftoper{\alpha} \subsquare (\dxd p)} + \maxroot{\shiftoper{\alpha} \subsquare q} - 2 \alpha (d-1)
\quad \text{(by induction)}
\\
& \leq  \maxroot{\shiftoper{\alpha} \subsquare p} + \maxroot{\shiftoper{\alpha} \subsquare q} - 2 \alpha d
\quad \text{(by Lemma~\ref{lem:p1q0})}.
\end{align*}
Lemma \ref{lem:p1q0} also tells us that equality is only achieved when $p = x^{d}$.
 
We now consider the case in which both $p$ and $q$ have degree $d$.
For polynomials $p$ and $q$ in $\rrpos{d}$, define
\[
  \phi (p,q) = 
 \maxroot{\shiftoper{\alpha} \subsquare p} + \maxroot{\shiftoper{\alpha} \subsquare q} - 2 \alpha d
 -   \maxroot{\shiftoper{\alpha} \subsquare  (p \recsum_{d} q)}.
\]
% For any $R > 0$, let $p$ and $q$ be the polynomials with all roots in $[0,R]$ that minimize $\phi (p,q)$.
% Since the set of degree $d$ polynomials with roots in $[0,R]$ is compact, there exist polynomials $p$ and $q$
%   on which the minimum is obtained.
% If $p$ and $q$ each have only one root, then Lemma \ref{lem:p1q1} implies $\phi (p,q) \geq 0$,
%     with equality only if one of them equals $x^{d}$.
% If not, let $p$ have at least two distinct roots.
% We may thus apply Corollary~\ref{cor:recPinch} to obtain
%   polynomials $\phat$ and $\ptil$.
We will prove that $\phi (p,q) \geq 0$ for all such polynomials.%, with equality only if one of them equals $x^{d}$.

Assume (for contradiction) that there exist polynomials $p, q$ with $\phi(p, q) < 0$ and let $[0, R]$ be an interval containing the roots of $p$ and $q$.
Again, $\phi$ is a continuous function (this time on the compact set $[0, R]^{2d}$) so let $p_0, q_0$ be a minimizer.
If {\em both} $p_0$ and $q_0$ have at most 1 distinct root, then Lemma \ref{lem:p1q1} implies $\phi (p_0,q_0) \geq 0$,
with equality only if one of them equals $x^{d}$ (a contradiction).
Hence we can assume without loss of generality that $p_0$ has at least 2 distinct roots and so Corollary \ref{cor:multPinch} provides polynomials $\phat$ and $\ptil$ with 
\[
\maxroot{\shiftoper{\alpha}\subsquare p_0} = \maxroot{\shiftoper{\alpha} \subsquare \ptil} = \maxroot{\shiftoper{\alpha} \subsquare  \phat}  
\]
and by Lemma~\ref{lem:max}
\[
\maxroot{\shiftoper{\alpha}\subsquare  (p_0 \recsum_{d} q_0)} \leq \max \{ \maxroot{\shiftoper{\alpha} \subsquare (\phat \recsum_{d} q_0)}, \maxroot{\shiftoper{\alpha} \subsquare (\ptil \recsum_{d} q_0)}
\]
which in turn implies $\phi(p_0, q_0) \geq \min{ \phi(\phat, q_0), \phi(\ptil, q_0) }$ with equality if and only if all are equal.
 Again equality cannot occur, since $\phi(p_0, q_0) < 0$ by assumption and $\phi(\phat, q_0) \geq 0$ by the inductive hypothesis and $\phi(p_0, q_0) > \phi(\phat, q_0)$ cannot occur for the same reason.
 But this implies $\phi(p_0, q_0) > \phi(\ptil, q_0)$, which contradicts the minimality of the pair $(p_0, q_0)$.
% 
% Let 
% \begin{align*}
% \beta 
% & = 
% \maxroot{\shiftoper{\alpha} \subsquare  p}
% + \maxroot{\shiftoper{\alpha} \subsquare q} - 2 \alpha d
% \\
% & =
% \maxroot{\shiftoper{\alpha} \subsquare  \phat }
% + \maxroot{\shiftoper{\alpha} \subsquare q} - 2 \alpha d
% \\
% & =
% \maxroot{\shiftoper{\alpha} \subsquare  \ptil }
% + \maxroot{\shiftoper{\alpha} \subsquare q} - 2 \alpha d
% .
% \end{align*}
% Now, assume by way of contradiction that $\phi (p,q) \leq  0$.
% This means that
% \[
%  \beta ' \defeq 
% \maxroot{\shiftoper{\alpha} \subsquare  (p \recsum_{d} q)}
%  \geq  \beta .
% \]
% 
% For $d \geq 2$, we may assume by induction that
% $\phi (\phat, q) > 0$ and thus
% \begin{equation}\label{eqn:pxq0}
%   (\shiftoper{\alpha} \subsquare (\phat  \recsum_{d} q)) (\beta) > 0.
% \end{equation}
% 
% As $\ptil = p - \phat $, we may conclude that
% \[
%     (\shiftoper{\alpha} \subsquare (\ptil  \recsum_{d} q)) (\beta ')
% =
%     (\shiftoper{\alpha} \subsquare (p  \recsum_{d} q))(\beta ')
% -
%   (\shiftoper{\alpha} \subsquare (\phat  \recsum_{d} q)) (\beta ')
% =
% -
%   (\shiftoper{\alpha} \subsquare (\phat  \recsum_{d} q)) (\beta ')
% < 0,
% \]
% and thus
% \[
%   \maxroot{\shiftoper{\alpha} \subsquare \dxd  \ptil } >  \beta' .
% \]
% But, this contradicts the assumption that $p$ and $q$ minimize $\phi (p,q)$
%   over polynomials with roots in $[0,R]$.

\end{proof}

\subsection{Ultraspherical Polynomials}\label{sec:cheby}

This section is devoted to the proof of Lemma~\ref{lem:p1q1}.
It is a consequence of the following lemma.

\begin{lemma}\label{lem:marcus}
For $d \geq 0$ and positive $\lambda$ and $\mu$,
\[
  \maxroot{\shiftoper{\alpha} \subsquare ((x-\lambda)^d \recsum_{d} (x-\mu)^d)}
  <
  \maxroot{\shiftoper{\alpha} (\subsquare (x-\lambda)^d \sqsum_{2d} \subsquare (x-\mu)^d))}.   
\]
\end{lemma}
Lemma~\ref{lem:p1q1} then follows from Theorem~\ref{thm:sqsumTrans}.
We will prove Lemma~\ref{lem:marcus} by showing that the polynomial
  on the left is a scaled Chebyshev polynomial of the second kind,
  and that the polynomial on the right is a Legendre polynomial with the same scaling.
We then appeal to known relations between the roots of these polynomials.

The Cheybshev and Legendre polynomials are both Ultraspherical (also called Gegenbauer) polynomials.  These are special cases of Jacobi polynomials.  It is known that their roots all lie between $-1$ and $1$ and that they are symmetric about zero (see Theorem 3.3.1 of \cite{szego1939orthogonal}).

%The Ultraspherical polynomials with parameter $\alpha$ are defined by the equations
%\[
%  C_0^{(\alpha)}(x) = 1, 
%  \quad
%  C_1^{(\alpha)}(x) = 2 \alpha x, 
%  \text{ and }
%  \quad
%  d C_d^{\alpha}(x) = 2 x (d + \alpha - 1) C_{d-1}^{\alpha}(x)
%    - (d + 2 \alpha - 2) C_{d-2}^{\alpha}(x).
%\]
The Ultraspherical polynomials with parameter $\alpha$ are defined by the generating function
\begin{equation}\label{eq:gegGenFunc}
\sum_n C^{(\alpha)}_n(x) t^n = \frac{1}{(1 - 2 x t + t^2)^{\alpha}}.
\end{equation}
Two special instances of these polynomials are 
\begin{enumerate}
\item the Legendre polynomials: $P_d (x) = C^{(1/2)}_d(x)$, and
\item the Chebyshev polynomials of the second kind: $U_{d}(x) = C^{(1)}_d(x)$.
\end{enumerate}

Stieltjes \cite{stieltjes1887racines} (see Theorem 6.21.1 of \cite{szego1939orthogonal}) established the following relation between 
  the zeros of the Chebyshev and Legendre polynomials.
\begin{theorem}\label{thm:stieltjes}
Let $\alpha_1 > \alpha_2 > \cdots > \alpha_{\lfloor d/2 \rfloor}$ be the positive
 roots of $U_{d}(x)$
 and let  $\beta_1 > \beta_2 > \cdots > \beta_{\lfloor d/2 \rfloor}$ be the positive
 roots of $P_{d}(x)$.
 Then, $\alpha_i < \beta_i$ for all $i$.
\end{theorem}

The relationship between the asymmetric additive convolution and the Chebyshev polynomials will make use of the generating function (\ref{eq:gegGenFunc}).
To aid in this endeavor, we recall the following well-known generalization of the {\em binomial theorem} (see, for example, \cite{wilf2005generatingfunctionology}).
\begin{theorem}\label{thm:binom}
The function $(1+z)^{-k}$ has the formal power series expansion
\[
\frac{1}{(1+z)^k} 
%= \sum_{i=0}^{\infty} \binom{-k}{i} z^i
= \sum_{i=0}^{\infty} \binom{k+i-1}{i}(-z)^i.
\]
\end{theorem}

\begin{lemma}\label{lem:cheby}
For $d \geq 0$ and positive $\lambda$ and $\mu$, 
\[
  \subsquare ((x-\lambda)^d \recsum_{d} (x-\mu)^d)  
  =  (\lambda \mu)^{d/2}
  U_{d} \left(\frac{x - (\lambda +\mu)}{2 \sqrt{\lambda \mu }} \right).
\]
\end{lemma}

\begin{proof}
By (\ref{eq:gegGenFunc}), we have
\[
\sum_{d} U_d(x)~t^d 
= \frac{1}{1 - 2 xt + t^2}
\]
and so 
\begin{align*}
\sum_{d} \sqrt{\lambda \mu}^d U_d \left(\frac{x - \lambda - 
\mu}{2 \sqrt{\lambda \mu}}\right) ~t^d 
&= \frac{1}{1 - (x - \lambda - \mu)t + \lambda \mu t^2}
\\&= \frac{1}{( 1- \lambda t)(1 - \mu t) - x t}
\\&= \frac{1}{( 1- \lambda t)(1 - \mu t)}\frac{1}{1 - \frac{x t}{(1 -\lambda t)(1 - \mu t)}}
\\&= \sum_{k \geq 0} \frac{x^{k} t^k}{(1-\lambda t)^{k+1}(1-\mu t)^{k+1}}
\\&= \sum_{i, j, k \geq 0} \binom{k+i}{i}\binom{k+j}{j} x^{k}(-\lambda)^i(-\mu)^j t^{i+j+k}
\end{align*}
so the coefficient of $t^d$ can be written (setting $\ell = d - k$)
\[
\sum_{\ell=0}^d x^{d-\ell} \sum_{i+j = \ell}\binom{d-j}{i} \binom{d - i}{j} (-\lambda)^i(-\mu)^j.
\]

On the other hand, the formula for the asymmetric convolution also gives us
\begin{align*}
(x-\lambda)^d \recsum_d (x-\mu)^d
&=
\sum_{k=0}^d x^{d-k} \sum_{i+j=k} \left(\frac{(d-i)!(d-j)!}{d!(d-i-j)!}\right)^2 \binom{d}{i}\binom{d}{j}(-\mu)^{i}(-\lambda)^{j}
\\&= 
\sum_{k=0}^d x^{d-k} \sum_{i+j=k} \frac{(d-i)!(d-j)!}{i!j!(d-i-j)!(d-i-j)!} (-\mu)^{i}(-\lambda)^{j}
\\&=
\sum_{k=0}^d x^{d-k} \sum_{i+j = k} \binom{d - i}{j}\binom{d-j}{i} x^{i}(-\lambda)^i(-\mu)^j.
\end{align*}
\end{proof}

The relationship between the symmetric additive convolution and the Legendre polynomials can be established by applying Theorem~\ref{thm:symmetricAdditiveRoots}

\begin{lemma}\label{lem:legendre}
  For $d \geq 0$ and positive $\lambda$ and $\mu$, 
%\[
%  (\subsquare (x-\lambda)^d )\sqsum_{2d} (\subsquare (x-\mu)^d)
%  =
%  (\lambda \mu)^{d/2} P_{d} \left(\frac{x - (\lambda +\mu)}{2 \sqrt{\lambda \mu }} \right).
%\]
\[
(\subsquare (x-\lambda)^d )\sqsum_{2d} (\subsquare (x-\mu)^d)
= 
\frac{(4 \sqrt{\lambda \mu})^{d}}{\binom{2d}{d}} P_d\left(\frac{x^2 - \lambda - 
\mu}{2 \sqrt{\lambda \mu}}\right) 
\]
\end{lemma}
  
\begin{proof}
We start by recalling one of the well-known formulas for Legendre polynomials \cite{szego1939orthogonal}:
\[
P_d(x) = \frac{1}{2^d} \sum_i \binom{d}{i}^2 (x-1)^i (x+1)^{d-i}.
\]
On the one hand we have
\begin{align*}
P_d\left(\frac{x^2 - \lambda - \mu}{2 \sqrt{\lambda \mu}} \right) 
&= 
\frac{1}{2^d} \sum_i \binom{d}{i}^2 
\left(\frac{x^2 - \lambda - \mu}{2 \sqrt{\lambda \mu}} - 1 \right)^i 
\left(\frac{x^2 - \lambda - \mu}{2 \sqrt{\lambda \mu}} + 1\right)^{d-i}
\\&=
\frac{1}{(4 \sqrt{\lambda\mu})^d} \sum_i \binom{d}{i}^2 (x^2 - (\sqrt{\lambda} + \sqrt{\mu})^2)^i 
(x^2 - (\sqrt{\lambda} - \sqrt{\mu})^2)^{d-i}.
\end{align*}

On the other hand, we can use Theorem~\ref{thm:symmetricAdditiveRoots} to calculate $(x^2-\lambda)^d \sqsum_{2d} (x^2-\mu)^d$.
There are four possible root sums that will appear: $\{ 
\pm(\sqrt{\lambda} + \sqrt{\mu}), \pm(\sqrt{\lambda} - \sqrt{\mu})
\}$.
Furthermore, it is easy to check that 
\begin{enumerate}
\item the $\pm (\sqrt{\lambda}+\sqrt{\mu})$ terms appear the same number of times in every pairing
\item the $\pm (\sqrt{\lambda}-\sqrt{\mu})$ terms appear the same number of times in every pairing
\item the probability of having $i$ copies of $(\sqrt{\lambda}+\sqrt{\mu})$ and $d-i$ copies of $(\sqrt{\lambda}-\sqrt{\mu})$ is $\binom{d}{i}^2/\binom{2d}{d}$
\end{enumerate}

Hence we have
\[
[(x^2-\lambda)^d \sqsum_{2d} (x^2-\mu)^d] 
= 
\frac{1}{\binom{2d}{d}}
\sum_{i = 0}^d \binom{d}{i}^2 (x^2-(\sqrt{\lambda} + \sqrt{\mu})^2)^i (x^2-(\sqrt{\lambda }
- \sqrt{\mu})^2)^{d-i}.
\]
\end{proof}

\begin{proof}[Proof of Lemma \ref{lem:marcus}]
We prove this in the case that $d$ is even.
The proof with $d$ odd is similar, except that there is one extra common term for the root at $0$.
Let $1 > \alpha_1 > \alpha_2 > \cdots > \alpha_{ d/2 }$ be the positive
  roots of $U_{d}(x)$ and let $\alpha_{d+1-i} = -\alpha_i$
  be the negative roots.
Similarly, let  $1 > \beta_1 > \beta_2 > \cdots > \beta_{d}$ be the
  roots of $P_{d}(x)$.
As $(\lambda + \mu) / 2 \sqrt{\lambda \mu} \geq 1$, the 
roots of 
\[
  \subsquare ((x-\lambda)^d \recsum_{d} (x-\mu)^d)
\]
are
\[
 \pm \left(  \sqrt{\lambda \mu} \beta_i + (\lambda + \mu) \right)^{1/2}
\]
and the roots of 
\[
  (\subsquare (x-\lambda)^d )\sqsum_{2d} (\subsquare (x-\mu)^d)
\]
are
\[
 \pm \left(  \sqrt{\lambda \mu} \alpha_i + (\lambda + \mu) \right)^{1/2}.
\]
By Theorem \ref{thm:stieltjes} the largest root of the first polynomial is larger than
  the largest root of the second, and for every $t$ larger than that root,

\begin{align*}
 \cauchy{t}{    \subsquare ((x-\lambda)^d \recsum_{d} (x-\mu)^d) }
 & = 
 \frac{1}{2d} \sum_{i = 1}^{d} 
  \frac{1}{t - \left(  \sqrt{\lambda \mu} \alpha_i + (\lambda + \mu) \right)^{1/2}}
  +
  \frac{1}{t + \left(  \sqrt{\lambda \mu} \alpha_i + (\lambda + \mu) \right)^{1/2}}
  \\
  & = \frac{1}{2d} \sum_{i = 1}^{d} 
    \frac{2t}{t^2 -  \sqrt{\lambda \mu} \alpha_i - (\lambda + \mu)}
  \\
    & = \frac{1}{2d} \sum_{i = 1}^{d/2} 
      \frac{2t}{t^2 - (\lambda + \mu) -  \sqrt{\lambda \mu} \alpha_i }
  +
   \frac{2t}{t^2 - (\lambda + \mu) + \sqrt{\lambda \mu} \alpha_i }
\\ 
& = \frac{1}{2d} \sum_{i = 1}^{d/2} 
\frac{4 t (t^2 - (\lambda + \mu))}{(t^2 - (\lambda + \mu))^2 -  \lambda \mu \alpha_i^2 }
\\
& < \frac{1}{2d} \sum_{i = 1}^{d/2} 
\frac{4 t (t^2 - (\lambda + \mu))}{(t^2 - (\lambda + \mu))^2 -  \lambda \mu \beta_i^2 }
  \\ &=   \cauchy{t}{    (\subsquare (x-\lambda)^d )\sqsum_{2d} (\subsquare (x-\mu)^d)   },
\end{align*}
where the last equality follows by reversing the previous reasoning.
This implies that 
\[
  \invcauchy{    \subsquare ((x-\lambda)^d \recsum_{d} (x-\mu)^d) }{w}
  <
  \invcauchy{ (\subsquare (x-\lambda)^d )\sqsum_{2d} (\subsquare (x-\mu)^d) }{w},
\]
where $w = 1/ \alpha d$.
\end{proof}

\bibliographystyle{alpha}
\bibliography{free}

%\appendix
%\section{List of Notation Used}

\end{document}